\documentclass[11pt]{article}
\usepackage{amsmath}
\usepackage{amsthm}
\usepackage{amsfonts}
\usepackage[latin5]{inputenc}
\usepackage{color}
\usepackage{enumitem}

\usepackage{fancyhdr}

\begin{document}
\numberwithin{equation}{section}
\newcounter{thmcounter}
\newcounter{Remarkcounter}
\newcounter{Defcounter}
\numberwithin{thmcounter}{section}
\newtheorem{Prop}[thmcounter]{Proposition}
\newtheorem{Corol}[thmcounter]{Corollary}
\newtheorem{theorem}[thmcounter]{Theorem}
\newtheorem{Lemma}[thmcounter]{Lemma}
\theoremstyle{definition}
\newtheorem{Def}[Defcounter]{Definition}
\theoremstyle{remark}
\newtheorem{Remark}[Remarkcounter]{Remark}
\newtheorem{Example}[Remarkcounter]{Example}

\newcounter{defcounter}
\setcounter{defcounter}{0}
\newcounter{defcounteri}
\setcounter{defcounteri}{0}
\newenvironment{BihEq}{%
\addtocounter{equation}{-1}
\refstepcounter{defcounter}
\renewcommand\theequation{BH\thedefcounter}
\begin{equation}}
{\end{equation}}
\newenvironment{BicEq}{%
\addtocounter{equation}{-1}
\refstepcounter{defcounteri}
\renewcommand\theequation{BC\thedefcounteri}
\begin{equation}}
{\end{equation}}

\newcommand{\diag}{\mathrm{diag}\ }
\newcommand{\DIV}{\mathrm{div}}
\title{A Classification of Biconservative Hypersurfaces in a Pseudo-Euclidean Space}
%%SHORT TITLE [Biconservative hypersurfaces]
\author{Abhitosh Upadhyay\footnote{Harish Chandra Research Institute, Chhatnag Road, Jhusi, Allahabad, India, 211019, e-mail: abhi.basti.ipu@gmail.com, abhitoshupadhyay@hri.res.in} and Nurettin Cenk Turgay \footnote{Istanbul Technical University, Faculty of Science and Letters, Department of  Mathematics, 34469 Maslak, Istanbul, Turkey, e-mail: turgayn@itu.edu.tr} \footnote{Corresponding Author}}
\maketitle
%\date{\today}
\begin{abstract}
In this paper, we study biconservative hypersurfaces of index 2 in $\mathbb E^{5}_{2}$. We give the complete classification of biconservative hypersurfaces with diagonalizable shape operator at exactly three distinct principal curvatures. We also give an explicit example of biconservative hypersurfaces with four distinct principal curvatures.

\end{abstract}

\section{Introduction}
Let $\mathbb E^m_s$ denote the pseudo-Euclidean $m$-space with the canonical 
pseudo-Euclidean metric tensor $g$ of index $s$ given by  
$$
 g=-\sum\limits_{i=1}^s dx_i^2+\sum\limits_{j=s+1}^m dx_j^2,
$$
where $(x_1, x_2, \hdots, x_m)$  is a rectangular coordinate system in $\mathbb E^m_s$. Consider an $n$-dimensional oriented submanifold  $M$ of $\mathbb E^m_s$ with Laplace operator $\Delta$ and mean curvature vector $H$. Let us consider  an isometric immersion $x:M\rightarrow\mathbb E^m_s$. $M$ is said to be biharmonic if $x$ satisfies $\Delta^2 x=0$. If the tangential part of $\Delta^2 x$ vanishes identically, then $M$ is said to be biconservative \cite{ChenKitapTotal2ndEd, Fu2013MinkBih}.

The well-known formula of Beltrami provides a relation between the position vector $x$ and mean curvature vector given by
\begin{equation}\label{LaplaceBeltramiFormula}
\Delta x=-nH. 
\end{equation}
Equation \eqref{LaplaceBeltramiFormula} implies that biharmonicity of $M$ is equivalent to have harmonic mean curvature vector, i.e., the equation $\Delta^2 x=0$ is satisfied if and only if  $\Delta H =0$. If $M$ is minimal, i.e., $H=0$, then from equation \eqref{LaplaceBeltramiFormula}, it is biharmonic. Bang-Yen Chen conjectured that if the ambient space is Euclidean then the converse of this statement is also true \cite{ChenOpenProblems, ChenRapor}. Although, Chen's biharmonic conjecture is still an open problem, there are a lot of results on submanifolds of Euclidean spaces which provide affirmative partial solutions to the conjecture \cite{BMO1,ChenOpenProblems,ChenRapor,Defever1996,Jiang1985,Dimitric1992,HsanisField, YuFu2014}. For example,  recently, Yu Fu has studied biharmonic hypersurfaces in $\mathbb E^5$ with at most three distinct principal curvatures and proved that the conjecture is true for this case \cite{YuFu2014}. 

 It has been observed that if the ambient space is a semi-Euclidean space then there may exist  non-minimal biharmonic submanifolds. For example, the surface given by 
$$x(u,v)=(\phi(u,v),u,v,\phi(u,v))$$
is a non-minimal biharmonic surface in $\mathbb E^4_1$ whereas $\phi$ is a particular chosen smooth function \cite{ChenIshikawa1991Bih}. One can also see \cite{ChenIshikawa1998Bih} for non-minimal surfaces in $\mathbb E^4_2$. Many geometers studied biharmonic submanifolds in semi-Euclidean spaces and obtained some interesting results in this direction \cite{Arvan1,Arvan2,Defever2006,CMO2,Papantoniou,Fu2013MinkBih}.

On the other hand, in order to understand the geometry of biharmonic submanifolds, \linebreak geometers have shown attention to study geometrical properties of biconservative submanifolds\linebreak and contributed accordingly \cite{chenMunteanu2013,FetcuOniciucPinheiro,TurgayHHypersurface,CMOP,MOR20141,
HsanisField}. It has been observed that some authors have called biconservative hypersurfaces as ``H-hypersurfaces" \cite{TurgayHHypersurface,HsanisField}. For example, Chen and\linebreak Munteanu showed that $\delta(2)$-ideal biconservative hypersurface in Euclidean space $\mathbb E^n$ is \linebreak isoparametric \cite{chenMunteanu2013}. Further, Caddeo et al.  classified biconservative surfaces in the 3-dimensional\linebreak Riemannian space form \cite{CMOP}. Montaldo, Oniciuc and Ratto  studied $SO(p+1)\times SO(q+1)$\linebreak-invariant and $SO(p+1)$-invariant biconservative hypersurfaces in Euclidean space by using framework of equivariant differential geometry \cite{MOR20141}. Recently, authors classify biconservative\linebreak surfaces in $S^n\times R$ and $H^n\times R$ in \cite{FetcuOniciucPinheiro}. Most recently, the second named author obtained complete classification of biconservative hypersurfaces with three distinct principal curvatures in Euclidean spaces \cite{TurgayHHypersurface}. It has been observed that Papantoniou et al. proved that a non-degenerate biharmonic hypersurface with index 2 in  $\mathbb E^4_2$ is minimal \cite{Papantoniou}. Most recently, in  \cite{PseudoYuxinDongandYe-Lin}, authors presented a brief summary of work on pseudo-Riemannian submanifolds which show immense possibilities to investigate in this direction.

  The purpose of the present paper is to study biconservative hypersurfaces of index $2$ in $\mathbb E^5_2$. In Section 1, we have presented a brief introduction of the previous work which has been done in this direction. In Section 2, we have collected the formulae and information which are useful in our subsequent sections. In Section 3, we have obtained our main results whereas in Section 4, we have presented the conclusion of the work and explicit examples to support our results. 
  
  The hypersurfaces which we are dealing are smooth and connected unless otherwise stated.

\section{Prelimineries}
Let $M$ be an oriented hypersurface in $\mathbb E^m_s$ with the unit normal vector field $N$ associated with the orientation of $M$. We denote Levi-Civita connections of $\mathbb E^m_s$ and $M$ by $\widetilde{\nabla}$ and $\nabla$, respectively. Then, the Gauss and Weingarten formulas are given by
\begin{eqnarray}
\label{MEtomGauss} \widetilde\nabla_X Y&=& \nabla_X Y + h(X,Y),\\
\label{MEtomWeingarten} \widetilde\nabla_X N&=& -S(X)
\end{eqnarray}
for all tangent vectors fields $X,\ Y$, where $h$ is the second fundamental form  and   $S$ is the shape operator of $M$, related by  $\left\langle h(X, Y),N \right\rangle = \left\langle S(X), Y \right\rangle.$ The Gauss and Codazzi equations are given, respectively, by
\begin{eqnarray}
\label{MinkGaussEquation}\label{GaussEq} R(X,Y,Z,W)&=&\langle h(Y,Z),h(X,W)\rangle-
\langle h(X,Z),h(Y,W)\rangle,\\
\label{MinkCodazzi} (\bar \nabla_X h )(Y,Z)&=&(\bar \nabla_Y h )(X,Z),
\end{eqnarray}
where  $R$ is the curvature tensor associated with connection $\nabla$ and  $\bar \nabla h$ is defined by
$$(\bar \nabla_X h)(Y,Z)=\nabla^\perp_X h(Y,Z)-h(\nabla_X Y,Z)-h(Y,\nabla_X Z).$$ 
We  denote complete pseudo-Riemannian manifolds of signature ($s, m-1$) by
\begin{eqnarray}
\mathbb S^{m-1}_{s}(r^2)&=&\{x\in\mathbb E^m_s: \langle x, x \rangle=r^{-2}\},\notag
\\
\mathbb H^{m-1}_{s-1}(-r^2)&=&\{x\in\mathbb E^{n+1}_1: \langle x, x \rangle=-r^{-2}\}\notag
\end{eqnarray}
where $r^{-2}$ and $-r^{-2}$ are respective constant sectional curvatures and $m>1$. $\mathbb S^{m-1}_{s} (r^2)$ and $\mathbb H^{m-1}_{s-1}(-r^2)$ are called the pseudo-Riemannian sphere and pseudo-Hyperbolic space, respectively.\linebreak
Moreover, if index $s=0$, then we denote $\mathbb H^{m-1}_{0}(-r^2)=\mathbb H^{m-1}(-r^2)$ and $\mathbb S^{m-1}_{0}(r^2)=\mathbb S^{m-1}(r^2)$.

\subsection{Biconservative Hypersurfaces in $\mathbb E^5_2$}
Let $x:M\rightarrow\mathbb E^5_2$ be an isometric immersion of an index $2$ hypersurface $M$ in $\mathbb E^5_2$ with the shape operator $S$. The mean curvature vector $\mathbf{H}$ of $M$ is defined by 
\begin{eqnarray}
\mathbf{H}=  \frac{1}{4}\hspace{.1cm}tr\hspace{.1cm} h
\end{eqnarray}
and $\mathbf{H}=HN$, where $H$ is the (first) mean curvature (function) of $M$. Consider an orthonormal base field $\{e_1,e_2,e_3,e_4\}$ such that $\langle e_i,e_i\rangle=\varepsilon_i,\ i=1,2,3,4$.
The Laplace operator $\Delta$ is defined as 
\begin{equation}\nonumber
\Delta = \sum\limits_{i=1}^4 \epsilon_i(\nabla_{e_{i}}e_{i} - e_{i}e_{i}).
\end{equation}
Moreover, gradient of a smooth function $f:M^{4}_{2}\rightarrow \mathbb R$ is given by 
\begin{equation}\nonumber
\nabla f =  \sum\limits_{i=1}^4 \epsilon_ie_{i}(f)e_{i}.
\end{equation}
By direct computation, using equation \eqref{LaplaceBeltramiFormula}, one can see that $M$ is biconservative hypersurface if and only if it satisfies the equation
\begin{BicEq}
\label{HSurfBicCondE52} S(\nabla H)=-2H\nabla H.\\
\end{BicEq}
\begin{Remark}
If $M$ has constant mean curvature $H$, then it is obvious that equation \eqref{HSurfBicCondE52} is satisfied. Therefore, we assume that $H$ is not constant, i.e., $\nabla H\neq0$. One can also refer \cite{ChenKitapTotal2ndEd,Lucas2011}.
\end{Remark}

\begin{Lemma}
Let $M$ be a hypersurface of index $2$ in $\mathbb E^5_2$ with $H$ as its (first) mean curvature. Assume that $\nabla H$ is not light-like. If $M$ is biconservative, then with respect to a suitable frame field $\{e_1=\frac{\nabla H}{\|\nabla H\|},e_2,e_3,e_4\}$, its shape operator $S$ has one of the following forms: 
\begin{align} \label{SOPCASES}
\begin{split}
\mbox{Case I. }S=\left( \begin{array}{cccc}
-2H &0&0&0\\
0 &k_2&0&0\\
0 &0&k_3&0\\
0 &0&0&k_4
\end{array}\right),\quad
\mbox{Case II. }&S=\left( \begin{array}{cccc}
-2H &0&0&0\\
0 &k_2&1&0\\
0 &0&k_2&0\\
0 &0&0&k_4
\end{array}\right),\\
\mbox{Case III. }S=\left( \begin{array}{cccc}
-2H &0&0&0\\
0 &k_2&-\nu&0\\
0 &\nu&k_2&0\\
0 &0&0&k_4
\end{array}\right),\quad
\mbox{Case IV. }&S=\left( \begin{array}{cccc}
-2H &0&0&0\\
0 &2H&0&0\\
0 &0&2H&-1\\
0 &1&0&2H
\end{array}\right),
\end{split}
\end{align}
for some smooth functions $k_2,k_3,k_4,\nu$. In Cases I and III, the induced metric $g_{ij}=g(e_i,e_j)=\langle e_i,e_j\rangle$ of $M$ is $g_{ij}=\varepsilon_i\delta_{ij}\in\{-1,1\}$, while in Cases II and IV, it is given by
$$g=\left( \begin{array}{cccc}
\varepsilon_1 &0&0&0\\
0 &0&-1&0\\
0 &-1&0&0\\
0 &0&0&-\varepsilon_1
\end{array}\right).$$
\end{Lemma}
\begin{proof}
If $M$ is biharmonic then equation \eqref{HSurfBicCondE52} is satisfied. Let $e_1=\frac{\nabla H}{\|\nabla H\|}$ which satisfies $Se_1=k_1e_1$ and we have $k_1=-2H$. Consider the  3-dimensional distribution $\hat D=(\mathrm{span}\{e_1\})^\perp$ having index 1 or 2 subject to being time-like or space-like of $e_1$, respectively, then we have $S(\hat D)\subset\hat D$. Therefore, the restriction $\hat S:\hat D\rightarrow$ $S$ into $\hat D$, is a self adjoint endomorphism expressed as
$$
\left( \begin{array}{ccc}
a_{11}&a_{12}&a_{13}\\
-a_{12}&a_{22}&a_{23}\\
-a_{13}&a_{23}&a_{33}
\end{array}\right)
$$
with respect to an arbitrary basis $\{f_2,f_3,f_4\}$ of $\hat D$ such that $-\langle f_2,f_2\rangle=\langle f_3,f_3\rangle=\langle f_4,f_4\rangle$. Therefore, the matrix representation  of $\hat S$ has one of the following forms
\begin{align} \label{SOPCASESForREST}
\begin{split}
\mbox{Case I. }\hat S=\left( \begin{array}{ccc}
k_2&0&0\\
0&k_3&0\\
0&0&k_4
\end{array}\right),\quad
\mbox{Case II. }&\hat S=\left( \begin{array}{ccc}
k_2&1&0\\
0&k_2&0\\
0&0&k_4
\end{array}\right),\\
\mbox{Case III. }\hat S=\left( \begin{array}{ccc}
k_2&-\nu&0\\
\nu&k_2&0\\
0&0&k_4
\end{array}\right),\quad
\mbox{Case IV. }&\hat S=\left( \begin{array}{ccc}
k_2&-\nu&0\\
0&k_2&-1\\
1&0&k_2
\end{array}\right)
\end{split}
\end{align}
 with respect to a suitable frame field $\{e_2,e_3,e_4\}$. Hence, the matrix representation of $S$ with respect to the frame field $\{e_1=\frac{\nabla H}{\|\nabla H\|}, e_2,e_3,e_4\}$ takes one of these four forms of the Lemma.
\end{proof}
%%%%%%%%%%%%%%%%%%%%%%%%%%%%%%%%%%%%%%%%%%%%%%%%%%%%%%%%%%%%%%%%%%%%%%%%%%%%%%%%%%%%%%%%%%%%%%%%%%%%%%%%%%%%%%%%%%%%%%%%%%%%%%%%%%%%%%%%%%%%%%%%%%%%%%%%%%%%%%%%%%%%%%%%%%%%%%%%%%%%%%%%%%%%%%%%%%%%%%%%%%%%%%%%%%%%%%%%%%%%%%%%%%%%%%%%%%%%%%%%%%%%%%%%%%%%%%%%%%%%%%%%%%%%%%%%%%%%%%%%%%%%%%%%%%%%%%%%%%%%%%%%%%%%%%%%%%%%%%%%%%%%%%%%%%%%%%%%%%%%%%%%%%%%%%%%%%%%%%%%%%%%%%%%%%%%%%%%%%%%%%%%%%%%%%%%%%%%%%%%%%%%%%%%%%%%%%%%%%%%%%%%%%%%%%%%%%%%%%%%%%%%%%%%%%%%%%%%%%%%%%%%%%%%%%%%%%%%%%%%%%%%%%%%%%%%%%%%%%%%%%%%%%%%%%%%%%%%%%%%%%%%%%%%%%%%%%%%%%%%%%%%%%%%%%%%%%%%%%%%%%%%%%%%%%%%%%%%%%%%%%%%%%%%%%%%%%%%%%%%%%%%%%%%%%%%%%%%%%%%%%%%%%%%%%%%%%%%%%%%%%%%%%%%%%%%%%%%%%%%%%%%%%%%%%%%%%%%%%%%%%%%%%%%%%%%%%%%%%%%%%%%%%%%%%%%%%%%%%%%%%%%%%%%%%%%%%%%%%%%%%%%%%%%%%%%%%%%%%%%%%%%%%%%%%%%%%%%%%%%%%%%%%%%%%%%%%%%%%%%%%%%%%%%%%%%%%%%%%%%%%%%%%%%%%%%%%%%%%%%%%%%%%%%%%%%%%%%%%%%%%%%%%%%%%%%%%%%%%%%%%%%%%%%%%%%%%%%%%%%%%%%%%%%%%%%%%%%%%%%%%%%%%%%%%%%%%%%%%%%%%%%%%%%%%%%%%%%%%%%%%%%%%%%%%%%%%%%%%%%%%%%%%%%%!
 %%%%%%%%%%%%%%%%%%%%%%%%%%%%%%%%%%%%%%%%%%%%%%%%%%%%%%%%%%%%%%%%%%%%%%%%%%%%%%%%%%%%%%%%%%%%%%%%%%%%%%%

\section{Classification of Biconservative Hypersurfaces with \\Diagonalizable Shape Operator}
A hypersurface is said to be isoparametric if its shape operator $S$ is diagonalizable and has constant eigen values (principal curvatures). In this section, we obtain classification of the biconservative hypersurfaces which has diagonalizable shape operator. 

\subsection{Connection forms}
Let $M$ be a hypersurface of index 2 with diagonalizable shape operator $S$ in $\mathbb E^{5}_{2}$. Consider an orthonormal frame field $\{e_{1}, e_{2}, e_{3}, e_{4}\}$ of $M$ consisting of its principal directions and $k_{1}, k_{2}, k_{3}, k_{4}$ are the corresponding principal curvatures whereas $\{\theta_{1}, \theta_{2}, \theta_{3}, \theta_{4}\}$ be the dual base field. The first structural equation of Cartan is given by
\begin{equation}\label{CartanFirstStr}
d\theta_{i} = \sum_{j = 1}^{4}\theta_{j}\wedge\omega_{ij}, \quad i = 1, 2, 3, 4,
\end{equation} 
where $\omega_{ij}$ denotes the connection forms satisfying $\omega_{ij}(e_{l}) = \langle\nabla_{e_{l}}e_{i}, e_{j}\rangle$, corresponding to the chosen frame field. Then, the Levi-Civita connection $\nabla$ of $M$ becomes
\begin{subequations}\label{LeviCivitaMEq1}
\begin{eqnarray}
\label{LeviCivitaMEq1a} \nabla_{e_{i}}e_{1} &=& \epsilon \omega_{12}(e_{i})e_{2} - \epsilon \omega_{13}(e_{i})e_{3} - \epsilon \omega_{14}(e_{i})e_{4},\\
\label{LeviCivitaMEq1b} \nabla_{e_{i}}e_{2} &=& \epsilon \omega_{12}(e_{i})e_{1} - \epsilon \omega_{23}(e_{i})e_{3} - \epsilon \omega_{24}(e_{i})e_{4},\\
\label{LeviCivitaMEq1c} \nabla_{e_{i}}e_{3} &=& - \epsilon \omega_{13}(e_{i})e_{1} - \epsilon\omega_{23}(e_{i})e_{2} - \epsilon \omega_{34}(e_{i})e_{4},\\
\label{LeviCivitaMEq1d} \nabla_{e_{i}}e_{4} &=& - \epsilon \omega_{14}(e_{i})e_{1} - \epsilon \omega_{24}(e_{i})e_{2} + \epsilon \omega_{34}(e_{i})e_{3}.
\end{eqnarray}
\end{subequations}
Now, from the Codazzi equation \eqref{MinkCodazzi}, we have
\begin{subequations}\label{CodazziFormaLL}
\begin{eqnarray}\label{CodazziForm1}e_{i}(k_{j}) = \epsilon_{j}\omega_{ij}(e_{j})(k_{i} - k_{j}),\\
\label{CodazziForm2}\omega_{jl}(e_{i})(k_{j} - k_{l}) = \omega_{il}(e_{j})(k_{i} - k_{l}), 
\end{eqnarray}
\end{subequations}
where indices $i,j,k$ are distinct and varies from 1 to 4.

Let $M$ is a biconservative hypersurface, i.e., $S$ and $H$ satisfies equation \eqref{HSurfBicCondE52}. This implies that $\nabla H$ is a principal direction with the corresponding principal curvature proportional to $H$ by a constant. Therefore, $e_{1} = \frac{\nabla H}{|\nabla H|}$ and $k_{1} + k_{2} + k_{3} + k_{4} = 4H$ gives that $k_{1} = - 2H$. Thus, we have
\begin{equation}\label{HangiHungiHsQQQQ1}
k_{2} + k_{3} + k_{4} = 6H.
\end{equation}
Since $e_{1}$ is proportional to $\nabla k_{1}$, we have
\begin{equation}\label{HangiHungiHs}
e_{2}(k_{1}) =  e_{3}(k_{1}) = e_{4}(k_{1}) = 0;\quad\quad e_{1}(k_{1})\neq 0.
\end{equation}
If $k_A=k_1,\ A=2,3,4$, then \eqref{CodazziForm1} leads to $e_{1}(k_{1})= 0$, which contradicts the equation \eqref{HangiHungiHs}. Therefore, we locally assume that $k_A+2H$ does not vanish for $A=2,3,4$.

\begin{Lemma}\label{LABELLEMMA1}
Let $M$ be a biconservative hypersurface of index 2 in $\mathbb E^{5}_{2}$ with diagonalizable shape operator. Then connection forms of $M$ satisfy
\begin{align}\label{ALLOFEQUATIONSCOMFOTM}
\begin{split}
\omega_{12}(e_{1}) = \omega_{12}(e_{3}) = \omega_{12}(e_{4}) =0,\\
\omega_{13}(e_{1}) = \omega_{13}(e_{2}) = \omega_{13}(e_{4}) =0,\\ 
\omega_{14}(e_{1}) = \omega_{14}(e_{2}) = \omega_{14}(e_{3}) = 0,\\
\omega_{23}(e_{1})(k_2-k_3) = \omega_{24}(e_{1}) (k_2-k_4)= \omega_{34}(e_{1}) (k_3-k_4)= 0.
\end{split}
\end{align}

\begin{proof}
Combining equations \eqref{CodazziForm1} and \eqref{HangiHungiHs}, we obtain $\omega_{1A}(e_{A})=0,\ A=2,3,4$. On the other hand, using $[e_A,e_B](k_1)=0$, we get $\omega_{1A}(e_B)=\omega_{1B}(e_A)$, for $A,B=2,3,4$ and $A\neq B$. 

  Therefore, equation \eqref{CodazziForm2} yields that $\omega_{1A}(e_B)=0$ for $i=A, j=B$, $l=1$ and \linebreak$\omega_{AB}(e_i)(k_A-k_B)=0$ for $i=A, j=1$, $l=B$.
\end{proof}
\end{Lemma}
\begin{Remark}\label{AnImpRemarkTOGIBE1}
Using equation \eqref{ALLOFEQUATIONSCOMFOTM}, we obtain $[e_1,e_A](k_1)=0$ which yields that $e_A(e_1(k_1))=0$, for $A=2,3,4$. Similarly, we have $e_A(e_1^2(k_1))=0$ and  $e_A(e_1^3(k_1))=0$, whereas $e_1^n(k_1)=\underbrace{e_1e_1\hdots e_1}_{n\mbox{ times}}(k_1).$
\end{Remark}

\begin{Remark}\label{CoordSystemS}
From Lemma \ref{LABELLEMMA1}, we have $\omega_{1i}=f_i\theta_i$ for some smooth functions $f_i$. Therefore, Cartan's first structural equation \eqref{CartanFirstStr} implies that $d\theta_{1} = 0$, i.e., $\theta_{1}$ is closed. But the Poincare Lemma implies that it is locally exact, i.e., there exists a local (orthogonal) coordinate system $(s, t, u, v)$ on a neighbourhood of $m\in M$ such that $\theta_{1} = ds$, from which we obtain $e_{1} =  \frac{\partial}{\partial s}$. Thus, we have $k_{1} = k_{1}(s)$, $k_{i} = k_{i}(s, t, u, v)$, $i = 2, 3, 4$. Since $k_{1}'\neq 0$ due to equation \eqref{HangiHungiHs}, the inverse function theorem implies that $s = s(k_{1})$ on a neighbourhood $N_{m}$ of $m$ in $M$ and we have $k_{i} = k_{i}(k_{1}, t, u, v)$. 
\end{Remark}

\subsection{Three distinct principal curvatures}
Let the hypersurface $M$ has exactly three distinct principal curvatures and $k_1=-2H$. Since we have $k_1\neq k_i$ for $i=2,3,4$, so without loss of generality, we may assume $k_2=k_3$. Thus, we have either $\langle e_2,e_2\rangle=-\langle e_3,e_3\rangle=\varepsilon$ or $\langle e_2,e_2\rangle=\langle e_3,e_3\rangle=\varepsilon$.

\begin{Lemma}\label{LemmaPrCurvatures}
Let $M$ be a biconservative hypersurface of index 2 in $\mathbb E^5_2$ with diagonalizable shape operator. If $M$ has three distinct principal directions $k_1,k_2,k_4$, then $e_i(k_j)=0$ for $i=2,3,4$ and $j=1,2,4$.
\end{Lemma}
\begin{proof}
We consider the case $\langle e_2,e_2\rangle=-\langle e_3,e_3\rangle=\varepsilon$. The result for the other case follows from an analogous computation.  Now, if one of $k_2$ and $k_4$ vanishes identically on an open subset $\widetilde{\mathcal O}$ of $M$, then equations
$$\left.k_2\right|_{\widetilde{\mathcal O}}=\left.k_4\right|_{\widetilde{\mathcal O}}=0$$ 
immediately follows from equations \eqref{HangiHungiHsQQQQ1} and \eqref{HangiHungiHs}. Therefore, we consider a component ${\mathcal O}$ of the open subset  $\{p\in M|k_2(p), k_4(p)\neq 0\}$. We define smooth functions $a,b,\alpha,\beta$ by
$$a=\left.k_2\right|_{\mathcal O},\quad b=\left.k_4\right|_{\mathcal O},\quad \alpha=\left.\omega_{12}(e_2)\right|_{\mathcal O},\quad \beta=\left.\omega_{14}(e_4)\right|_{\mathcal O}.$$ Then, equation \eqref{HangiHungiHsQQQQ1} becomes
\begin{eqnarray}
\label{Equation1}2k_2+k_3=6H. 
\end{eqnarray}
Furthermore, the Codazzi equation \eqref{CodazziForm1} takes the form
\begin{eqnarray}
\label{Codazzi1} e_1(a)&=&-\varepsilon(2H+a),\\
\label{Codazzi2} e_1(b)&=&\varepsilon(2H+b)
\end{eqnarray}
and the Gauss equation \eqref{MinkGaussEquation} gives
\begin{eqnarray}
\label{Gauss1} e_1(\alpha)&=&2aH-\varepsilon\alpha^2,\\
\label{Gauss2}e_1(\beta)&=&-2bH+\varepsilon\beta^2.
\end{eqnarray}
By applying $e_1$ to equation \eqref{Equation1} and using equations \eqref{Codazzi1} and \eqref{Codazzi2}, we get
\begin{eqnarray}
\label{Equation2} -2 \epsilon  \alpha (a+2 H)+\epsilon  \beta (b+2 H)=6 e_1(H).
\end{eqnarray}
By applying $e_1$ to equation \eqref{Equation2} twice and using equations \eqref{Codazzi1}-\eqref{Gauss2} in the obtained \linebreak equations, we get
\begin{align}\label{Equation3}
\begin{split}
\alpha ^2 (4 a+8 H)-4 \epsilon  a^2 H-8 \epsilon  a H^2+&\beta ^2 (2 b+4 H)-4 \epsilon  b H^2-2 \epsilon  b^2 H\\
&-4 \epsilon  \alpha  e_1(H)+2 \epsilon  \beta  e_1(H)=6 e_1^2(H)
\end{split}
\end{align}
and
\begin{align}\label{Equation4}
\begin{split}
\beta ^3 (6 \epsilon  b+12 \epsilon  H)+\alpha ^3 (-12 \epsilon  a-24 \epsilon  H)+12 \alpha ^2 e_1(H)+6 \beta ^2 e_1(H)\\
+\alpha  \left(56 a H^2+24 a^2 H-4 \epsilon  e_1^2(H)+16 H^3\right)+\beta  \left(-28 b H^2-12 b^2 H+2 \epsilon  e_1^2(H)-8 H^3\right)\\
-4 \epsilon  a^2 e_1(H)-24 \epsilon  a H e_1(H)-2 \epsilon  b^2 e_1(H)-12 \epsilon  b H e_1(H)=6 e_1^3(H).
\end{split}
\end{align}
From equations \eqref{Equation1} and \eqref{Equation2}, we get
$$\beta=-\frac{\alpha  (a+2 H)+3 \epsilon  e_1(H)}{a-4 H}.$$
We use $\beta$ in equations \eqref{Equation3} and \eqref{Equation4} to get 
\begin{align}\label{Equation3a}
\begin{split}
-a \left(e_1^2(H)+5 \epsilon  \alpha  e_1(H)+4 H \alpha ^2+48 \epsilon  H^3\right)+16 \epsilon  a^2 H^2-2 \epsilon  a^3 H-7 e_1(H)^2&\\-H \left(6 \epsilon  \alpha  e_1(H)-4 e_1^2(H)\right)-8 H^2 \alpha ^2+64 \epsilon  H^4&=0
\end{split}
\end{align}
and
\begin{align}\label{Equation4a}
\begin{split}
48 a^3 H (\epsilon  e_1(H)-2 H \alpha )+2 a^4 (6 H \alpha -\epsilon  e_1(H))+a^2 (-e_1^3(H)-\epsilon  \alpha  e_1^2(H)+21 \alpha ^2 e_1(H)&\\
-388 \epsilon  H^2 e_1(H)+24 \epsilon  H \alpha ^3+144 H^3 \alpha )+a (1312 \epsilon  H^3 e_1(H)&\\
+e_1(H) (60 \epsilon  \alpha  e_1(H)-e_1^2(H))+H (8 H^{(3)}+6 \epsilon  \alpha  e_1^2(H)+60 \alpha ^2 e_1(H))+24 \epsilon  H^2 \alpha ^3+384 H^4 \alpha )&\\
-1600 \epsilon  H^4 e_1(H)+63 e_1(H)^3+4 H e_1(H) (e_1^2(H)+30 \epsilon  \alpha  e_1(H))&\\
+4 H^2 (-4 e_1^3(H)-2 \epsilon  \alpha  e_1^2(H)+27 \alpha ^2 e_1(H))-48 \epsilon  H^3 \alpha ^3-768 H^5 \alpha&=0 
\end{split}
\end{align}
After a long computation, we remove $\alpha$ from these equations and get a non-trivial 14th degree polynomial equation of $a$, expressed as  
$$\sum\limits_{i=0}^14K_i(H,e_1(H),e_1e_1(H),e_1e_1e_1(H))a^i=0,$$
with coefficients $K_{14}=387200 \epsilon  H^6 e_1(H)^2$ and $K_{13}=-7040000 \epsilon  H(u)^7 e_1(H)^2$. This yields that the principal curvature $k_2$ has the form $k_2=a(H,e_1(H),e_1^2(H),e_1^3(H))$ in $\mathcal O$. Taking into account Remark \ref{AnImpRemarkTOGIBE1}, we obtain $e_A(k_2)=0,A=2,3,4$ on $\mathcal O$ whereas due to equation \eqref{Equation1}, we also have $e_A(k_4)=0$. An analogous computation yields the same result if $\langle e_2,e_2\rangle=\langle e_3,e_3\rangle=\varepsilon$. Hence, the proof is completed.
\end{proof} 
Next, we would like to give the following result obtained from the above Lemma \ref{LemmaPrCurvatures}.
\begin{Corol}\label{CORRLODLASD}
Let $M$ be a biconservative hypersurface of index $2$ in $\mathbb E^5_2$ with exactly 3 distinct\linebreak principal curvatures $k_1,k_2=k_3,k_4$. Then, the corresponding principal directions $e_1=\partial_s,e_2,e_3,e_4$ satisfy
\begin{subequations}
\begin{eqnarray}
\nabla_{e_1}e_i=0,\quad i=1,2,3,4,\\
\label{EQ317B}\nabla_{e_A}e_1=\epsilon_A\omega_{1A}(e_A)e_A,\quad A=2,3,4,\\
\label{EQ317C}\nabla_{e_{B}}e_{4} =  0,\quad B=2,3,\\
\label{EQ317D}\nabla_{e_{4}}e_{4} =  -\epsilon_1\omega_{14}(e_4)e_1,\\
\nabla_{e_{2}}e_{2} = -\epsilon_1 \omega_{12}(e_{2})e_{1} + \epsilon_3 \omega_{23}(e_{2})e_{3}, \quad \nabla_{e_{2}}e_{3} = - \epsilon_2 \omega_{23}(e_{2})e_{2},\\
\nabla_{e_{3}}e_{2} = - \epsilon_3 \omega_{23}(e_{3})e_{3}, \quad \nabla_{e_{3}}e_{3} = -\epsilon_1 \omega_{13}(e_{3})e_{1} + \epsilon_2 \omega_{23}(e_{3})e_{2},  \\
\nabla_{e_4}e_{2} =  - \epsilon_3 \omega_{23}(e_4)e_{3},\quad \nabla_{e_4}e_{3} =  - \epsilon_2\omega_{23}(e_4)e_{2}.
\end{eqnarray}
\end{subequations}
Moreover, components of connection forms satisfy 
\begin{equation}\label{omega2omega3REL}
e_A(\omega_{12}(e_2))=0 \quad \mbox{ and }\quad \varepsilon_2\omega_{12}(e_2)=\varepsilon_3\omega_{13}(e_3).
\end{equation}
\end{Corol}

\begin{proof}
Considering Lemma \ref{LemmaPrCurvatures} and Codazzi equations \eqref{CodazziFormaLL}, we obtain $\omega_{24}(e_{2})=\omega_{24}(e_4)=\omega_{34}(e_4)=\omega_{34}(e_2)=\omega_{24}(e_3)=0$ as $k_2=k_3$. Combining these equations with Lemma \ref{LABELLEMMA1} and equation \eqref{LeviCivitaMEq1}, we obtain the required equations. Furthermore, equation \eqref{omega2omega3REL} follows from equation \eqref{CodazziFormaLL}.
\end{proof}
%%%%%%%%%%%%%%%%%%%%%%%%%%%%%%%%%%%%%%%%%%%%%%%%%%%%%%%%%%%%%%%%%%%%%%%%%%%%%%%%%%%%%%%%%%%%%%%%%%%%%%%%%%%%%%%%%%%%%%%%%%%%%%%%%%%%%%%%%%%%%%%%%%%%%%%%%%%%%%%%%%%%%%%%%%%%%%%%%%%%%%%%%%%%%%%%%%%%%%%%%%%%%%%%%%%%%%%%%%%%%%%%%%%%%%%%%%%%%%%%%%%%%%%%%%%%%%%%%%%%%%%%%%%%%%%%%%%%%%%%%%%%%%%%%%%%%%%%%%%%%%%%%%%%%%%%%%%%%%%%%%%%%%%%%%%%%%%%%%%%%%%%%%%%%%%%%%%%%%%%%%%%%%%%%%%%%%%%%%%%%%%%%%%%%%%%%%%%%%%%%%%%%%%%%%%%%%%%%%%%%%%%%%%%%%%%%%%%%%%%%%%%%%%%%%%%%%%%%%%%%%%%%%%%%%%%%%%%%%%%%%%%%%%%%%%%%%%%%%%%%%%%%%%%%%%%%%%%%%%%%%%%%%%%%%%%%%%%%%%%%%%%%%%%%%%%%%%%%%%%%%%%%%%%%%%%%%%%%%%%%%%%%%%%%%%%%%%%%%%%%%%%%%%%%%%%%%%%%%%%%%%%%%%%%%%%%%%%%%%%%%%%%%%%%%%%%%%%%%%%%%%%%%%%%%%%%%%%%%%%%%%%%%%%%%%%%%%%%%%%%%%%%%%%%%%%%%%%%%%%%%%%%%%%%%%%%%%%%%%%%%%%%%%%%%%%%%%%%%%%%%%%%%%%%%%%%%%%%%%%%%%%%%%%%%%%%%%%%%%%%%%%%%%%%%%%%%%%%%%%%%%%%%%%%%%%%%%%%%%%%%%%%%%%%%%%%%%%%%%%%%%%%%%%%%%%%%%%%%%%%%%%%%%%%%%%%%%%%%%%%%%%%%%%%%%%%%%%%%%%%%%%%%%%%%%%%%%%%%%%%%%%%%%%%%%%%%%%%%%%%%%%%%%%%%%%%%%%%%%%%%%%%%%%%!
 %%%%%%%%%%%%%%%%%%%%%%%%%%%%%%%%%%%%%%%%%%%%%%%%%%%%%%%%%%%%%%%%%%%%%%%%%%%%%%%%%%%%%%%%%%%%%%%%%%%%%%%%%%%%%%%%%%%%%%%%%%%%%%%%%%%%%%%%%%%%%%%%%%%%%%%%%%%%%%%%%%%%%%%%%%%%%%%%%%%%%%%%%%%%%%%%%%%%%%%%%%%%%%%%%%%%%%%%%%%%%%%%%%%%%%%

Let us consider the distributions given by
\begin{equation}
D(m) = \mathrm{span}\{e_{2}|_{m}, e_{3}|_{m}\}, \hspace{.2cm} D'(m) = \mathrm{span}\{e_{4}\}. 
\end{equation}
The following lemma follows from a direct computation using Corollary \ref{CORRLODLASD}.
\begin{Lemma}
The distributions $D$ and $D'$ are involutive. 
\end{Lemma}
Let $x$ be a local parametrization of a neighbourhood of $m$ in $M$ where $m\in M$. Consider the integral submanifold of $D$ passing through $m$. We put $f_{1} = e_{2}|_{\tilde{M}}$, $f_{2} = e_{3}|_{\tilde{M}}$, $f_{3} = e_{1}|_{\tilde{M}}$, $f_{4} = e_{4}|_{\tilde{M}}$, $f_{5} = N|_{\tilde{M}}$ as local orthonormal frame field, consisting of restriction of vector fields $e_{2}, e_{3}, e_{1}, e_{4}, N$ to $\tilde{M}$. Then $\{f_{1}, f_{2}\}$ spans the tangent space of $\tilde{M}$ and $\{f_{3}, f_4, f_{5}\}$ spans the normal space of $\tilde{M}$ in $\mathbb E^{5}_{2}$. We denote $\delta_x=\langle f_x,f_x\rangle, x=1,2,\hdots, 5$  which obviously implies that $\delta_1=\varepsilon_2,\ \delta_2=\varepsilon_3,\  \delta_3=\varepsilon_1,\ \delta_4=\varepsilon_4$ and $\delta_5=1$.

It is observed that equation \eqref{EQ317C} yields that $f_4$ is a constant normal vector field on $\tilde{M}$ whereas \eqref{omega2omega3REL} yields that $\omega_{12}(e_{2})$ and $\omega_{13}(e_{3})$ are constant on $\tilde M$. Moreover, we have $\widetilde\nabla_{f_i}f_3= \alpha_0 f_i$ for a constant $\alpha_0$. Furthermore, by the Lemma \ref{LemmaPrCurvatures}, we have $k_2=k_3=\beta_0$ on $\widetilde M$. Hence, we have the following:
\begin{Lemma}\label{LemmaShapeOPS}
Let $f_{1}, f_{2}, f_{3}, f_{4}$ and $f_{5}$ are the vector fields defined above. Then $f_{3}, f_{5}$ are parallel vector fields whereas $f_{4}$ is a constant normal vector field. The matrix representations of the shape operators $\tilde{S}_{f_{3}}$ and $\tilde{S}_{f_{5}}$ are given by
\begin{equation}\label{Shatf3f5}
\tilde{S}_{f_{3}} = \epsilon_{1}\alpha_{0}I,\hspace{.4 cm} \tilde{S}_{f_{5}} = \beta_{0}I
\end{equation}
where $I$ is the identity operator acting on the tangent bundle of $\tilde M$.

\end{Lemma}

From Lemma \ref{LemmaShapeOPS}, we have the following proposition \cite[Lemma 4.2]{MagidIsop}:
\begin{Prop}\label{PROPSTUVCOORDS}
Let $M$ be a biconservative hypersurface of index 2 in $\mathbb E^{5}_{2}$ with diagonalizable shape operator. Then there exists a local coordinate system $(s, t, u, v)$ such that 
\begin{equation}
e_{1} = \frac{\partial}{\partial s}, \hspace{.3 cm} e_{2} = \frac{1}{E_{1}}\frac{\partial}{\partial t},\hspace{.3 cm} e_{3} = \frac{1}{E_{2}}\frac{\partial}{\partial u},\hspace{.3 cm} e_{3} = \frac{1}{E_{3}}\frac{\partial}{\partial v}.
\end{equation}
\end{Prop}
\begin{proof}
Let $D = span\{e_{2}|_{m}, e_{3}|_{m}\}$ be a given distribution at any point $m\in M$. Assuming $D^{\perp} = \{e_{3}|_{m}\}$, we get $D$ and $D^{\perp}$, both are complementary to each other at every point of $M$ and they are also involutive. Using (\cite{{KobayashiBook}}, Lemma in page 182), we find that there is orthogonal local coordinate system $(s, t, u, v)$ on $M$ such that $s$ is the coordinate function given in Remark \ref{CoordSystemS}, i.e., $e_{1}=\partial_{s}$. Also, we have $\mathrm{span}\{\partial_t,\partial_u,\partial_v\}=D\oplus D^{\perp}$, where $D\oplus D^{\perp}$ denotes direct sum of the involutive distributions $D$ and $D^{\perp}$. Since distributions $D$ and $D^{\perp}$ are involutive, we may re-define $t,u,v$ such that $\mathrm{span}\{\partial_t,\partial_u\}=D$ and $e_4=(E_3)^{-1}\partial_v$ for a smooth non-vanishing function $E_3$ \cite[Lemma 4.2]{MagidIsop}.
\end{proof}

Now, we obtain the local parametrization of biconservative hypersurfaces of index 2 in $\mathbb E^{5}_{2}$ with diagonalizable shape operator.
\begin{Prop}
Let $M$ be a biconservative hypersurface of index 2 in $\mathbb E^{5}_{2}$ such that, with respect to the orthonormal frame field $\{e_1=\frac{\nabla H}{\|\nabla H\|}, e_{2}, e_{3}, e_{4}\}$, its shape operator is given by
\begin{center}
$S =
 \begin{pmatrix}
  k_{1} & 0 & 0 & 0\\
  0 & k_{2} &  0 &  0\\
  0 &  0 &  k_{2} &  0\\
  0 &  0 &  0 & k_{4}
 \end{pmatrix}$,
\end{center}  
where $k_{1} = -4H$ and $k_{2} \neq k_{4}$. Then, $M$ has the following local parametrization
\begin{equation}\label{LOCALARAMETRIZ}
x(s, t, u, v) = \phi_{1}(s)\Theta_{1}(t, u) + \phi_{2}(s)\Theta_{2}(v) + \Gamma(s)
\end{equation} 
for some $\mathbb E^{5}_{2}$-valued functions $\Theta_{1}$, $\Theta_{2}$, $\Gamma$ and some smooth real valued functions $\phi_{1}$, $\phi_{2}$.
\begin{proof}
From Lemma \ref{LemmaPrCurvatures} and Codazzi equations \eqref{CodazziFormaLL}, we obtain 
\begin{equation}
e_{i}(\omega_{1j}(e_{j})) = 0,\hspace{.5 cm} i, j = 2, 3, 4.
\end{equation}
Therefore, we put
\begin{equation}
\varepsilon_2\omega_{12}(e_{2}) = \varepsilon_3\omega_{13}(e_{3}) = \alpha (s),\hspace{.5 cm} \varepsilon_4\omega_{14}(e_{4}) = \beta(s)
\end{equation}
for some smooth functions $\alpha$ and $\beta$. From Corollary \ref{CORRLODLASD} and the coordinate system given in the Proposition \ref{PROPSTUVCOORDS}, we obtain
\begin{equation}
x_{st} = \alpha(s)x_{t},\hspace{.2 cm} x_{su} = \alpha(s)x_{u},\hspace{.2 cm} x_{sv} = \beta(s)x_{v}.
\end{equation}
Integrating these equations, we obtain the result.
\end{proof}
\end{Prop}
\begin{Corol}\label{COROLINTSUB}
A slice given by $y(t, u) = x(c_{1}, t, u, c_{2})$ is an integral submanifold of the given distribution $D$.
\end{Corol}
\begin{Corol}\label{COROLINTCURVE}
The curve $\alpha(v) = x(c_{1}, c_{2}, c_{3}, v)$ is an integral curve of $e_{4}$.
\end{Corol}
%%%%%%%%%%%%%%%%%%%%%%%%%%%%%%%%%%%%%%%%%%%%%%%%%%%%%%%%%%%%%%%%%%%%%%%%%%%%%%%%%%%%%%%%%%%%%%%%%%%%%%%%%%%%%%%%%%%%%%%%%%%%%%%%%%%%%%%%%%%%%%%%%%%%%%%%%%%%%%%%%%%%%%%%%%%%%%%%%%%%%%%%%%%%%%%%%%%%%%%%%%%%%%%%%%%%%%%%%%%%%%%%%%%%%%%%%%%%%%%%%%%%%%%%%%%%%%%%%%%%%%%%%%%%%%%%%%%%%%%%%%%%%%%%%%%%%%%%%%%%%%%%%%%%%%%%%%%%%%%%%%%%%%%%%%%%%%%%%%%%%%%%%%%%%%%%%%%%%%%%%%%%%%%%%%%%%%%%%%%%%%%%%%%%%%%%%%%%%%%%%%%%%%%%%%%%%%%%%%%%%%%%%%%%%%%%%%%%%%%%%%%%%%%%%%%%%%%%%%%%%%%%%%%%%%%%%%%%%%%%%%%%%%%%%%%%%%%%%%%%%%%%%%%%%%%%%%%%%%%%%%%%%%%%%%%%%%%%%%%%%%%%%%%%%%%%%%%%%%%%%%%%%%%%%%%%%%%%%%%%%%%%%%%%%%%%%%%%%%%%%%%%%%%%%%%%%%%%%%%%%%%%%%%%%%%%%%%%%%%%%%%%%%%%%%%%%%%%%%%%%%%%%%%%%%%%%%%%%%%%%%%%%%%%%%%%%%%%%%%%%%%%%%%%%%%%%%

Next, we obtain integral submanifolds (surfaces) of the involutive distribution $D$ and integral curves of the $1$-dimensional involutive distribution $e_4$.

\begin{Prop}\label{PROPINSUBMOFD}
Let $M$ be a biconservative  hypersurface of index 2 in $\mathbb E^{5}_{2}$ with diagonalizable\linebreak shape operator and $\tilde{M}: x(s_{0}, t, u, v_{0}) = y(t, u)$ be the integral submanifold of the distribution $D$, passing through a point $p\in M$. If $M$ has three distinct principal curvatures, then $\tilde{M}$ is congruent to one of the following surfaces given by:
\renewcommand{\theenumi}{(\roman{enumi})}
\begin{enumerate}
\item\label{CASETOTGEO} A totally geodesic surface of $\mathbb E^5_2$, i.e., a non-degenerated 2-plane;

\item\label{LabCaseH20} A hyperbolic surface lying on a Lorentzian 3-plane, i.e., $\tilde{M}^{2}\simeq \mathbb H^{2}(-r^{2})\subset \mathbb E^{3}_{1}\subset \mathbb E^{4}_{1}\subset \mathbb E^{5}_{2}$, given by 
\begin{equation}\nonumber
y(t, u) = (0, r\mathrm{cosh} t, r\mathrm{sinh} t \cos u, r\mathrm{sinh} t \sin u, 0);
\end{equation}

\item\label{LabCaseS20}  A usual sphere lying on an Euclidean 3-plane, i.e., $\tilde{M}^{2}\simeq \mathbb S^{2}(r^{2})\subset \mathbb E^{3}\subset \mathbb E^{4}_{1}\subset \mathbb E^{5}_{2}$, given by
\begin{equation}\nonumber
y(t, u) = (0, 0, r\cos t, r\sin t \cos u, r\sin t \sin u);
\end{equation}

\item\label{LabCaseDegenM20}  A space-like surface lying on a degenerated hyperplane $\Pi$, i.e., $\tilde{M}^{2}\subset \Pi\subset \mathbb E^{4}_{1}\subset \mathbb E^{5}_{2}$, given by
\begin{equation}\label{LabCaseDegenM20SURF}
y(t, u) = ( A t^{2} + A u^{2},0, t, u, A t^{2} + A u^{2});
\end{equation}

\item\label{LabCaseH21}  A Lorentzian space-form lying on a 3-plane, i.e., $\tilde{M}^{2}_{1}\simeq \mathbb H^{2}_{1}(-r^{2})\subset \mathbb E^{3}_{2}\subset \mathbb E^{4}_{2}\subset \mathbb E^{5}_{2}$, given by 
\begin{equation}\nonumber
y(t, u) = (r\mathrm{cosh} t \sin u, r\mathrm{cosh} t \cos u, r\mathrm{sinh} t, 0, 0);
\end{equation}

\item\label{LabCaseDegenM21}  A Lorentzian surface lying on a degenerated hyperplane $\widetilde\Pi$, i.e., 
$\tilde{M}^{2}_{1}\subset \widetilde\Pi\subset \mathbb E^{4}_{2}\subset \mathbb E^{5}_{2}$, given    by 
\begin{equation}\label{LabCaseDegenM21SURF}
y(t, u) = (A t^{2} - A u^{2}, t, u,0, A t^{2} - A u^{2});
\end{equation}

\item\label{LabCaseS21}  A Lorentzian space-form lying on a Lorentzian 3-plane, i.e., $\tilde{M}^{2}_{1}\simeq \mathbb S^{2}_{1}(r^{2})\subset \mathbb E^{3}_{1}\subset \mathbb E^{4}_{2}\subset \mathbb E^{5}_{2}$, given by 
\begin{equation}\nonumber
y(t, u) = (0, r\mathrm{sinh} t, r\mathrm{cosh}t\cos u, r\mathrm{cosh} t \sin u, 0).
\end{equation}

\item\label{LabCaseS22}  A space-form lying on a 3-plane, i.e., $\tilde{M}^{2}_{2}\simeq\mathbb S^2_2(r^2)\subset \mathbb E^{3}_{2}\subset \mathbb E^{4}_{2}\subset \mathbb E^{5}_{2}$, given by
\begin{equation}\nonumber
y(t, u) = (r\mathrm{sinh} t cos u, r\mathrm{sinh} t \sin u, r\mathrm{cosh} u, 0, 0);
\end{equation}
\end{enumerate}
\end{Prop}

\begin{proof}
Let $\tilde M$ be the integral submanifold of the distribution $D$ passing through a point $p\in M$. If $p\in\widetilde{\mathcal O}$, then by a direct computation, one can obtain that the second fundamental form of $\tilde M$ vanishes identically, i.e., it is a totally geodesic surface of $\mathbb E^5_2$, where $p\in\widetilde{\mathcal O}$ is the interior of $\{p\in M|k_2(p)=0\}$. Thus, we have the case \ref{CASETOTGEO} of the proposition.

Now, we assume $p\not\in\widetilde{\mathcal O}$ and consider the local orthonormal frame field $\{f_{1}, f_{2}; f_{3}, f_{4}, f_{5}\}$ described above. From Lemma \ref{LemmaShapeOPS}, we have
\begin{equation}\label{CovDerf35Eq1}
\widetilde\nabla_{f_i}f_3=-\alpha_0f_i,\quad\quad \widetilde\nabla_{f_i}f_5=-\beta_0f_i,
\end{equation}
where $\alpha_0$ and $\beta_0$ are constants defined above. Moreover, $\tilde M$ lies on a hyperplane $\Sigma^4_r$  with index  $r=2$ or $r=1$ depending upon $\delta_4=1$ or $\delta_4=-1$ whereas $f_4$ is a constant normal vector field of $\tilde M$. Before considering these two cases separately, we define another normal vector field 
$$\zeta=\beta_0f_3-\alpha_0f_5$$
which is constant because of equation \eqref{CovDerf35Eq1}.

\textbf{Case I}. $\delta_4=1$. In this case, the index of the induced metric of $\tilde M$ is either 1 or 2 subject to $\delta_3=-1$ or $\delta_3=1$, respectively. 

\textbf{Case Ia}. $\delta_4=1$, $\delta_3=1$. In this case, $\zeta$ is a space-like constant vector field normal to $\tilde M$. Therefore, $\tilde M$ lies on a 3-plane $\mathbb E^3_2\simeq\Pi^3_2\subset\Sigma^4_2$ of $\mathbb E^5_2$. Furthermore, the normal vector field of $\tilde M$ in $\Pi$ given by $\eta=\alpha_0f_3+\beta_0f_5$, satisfies $\widetilde\nabla_{f_i}\eta=-(\alpha_0^2+\beta_0^2)f_i$. Therefore, $M$ is congruent to an isoparametric surface in $\mathbb E^3_2$ with index 2. Hence, we have the case \ref{LabCaseS22} of the proposition.

\textbf{Case Ib} $\delta_4=1$, $\delta_3=-1$. In this case, the induced metric of $M$ is Lorentzian. However, we have two subcases regarding to causality of $\zeta$.

\textbf{Case Ib.(i)} $\zeta$ is not light-like. In this case, $\tilde M$ lies on a 3-plane $\Pi^3_r\subset\Sigma^4_2$ of index $r$ which is either 2 or 1 regarding to being space-like or time-like of $\zeta$, respectively. If $r=2$, then a similar argument to Case Ia yields that $M$ is congruent to $\mathbb H^2_1(-r^2)$ which gives case \ref{LabCaseH21} of the proposition. On the other hand, if $r=1$, then  $M$ is congruent to $\mathbb S^2_1(r^2)$. Thus, we have the case \ref{LabCaseS21} of the proposition.

\textbf{Case Ib.(ii)} $\zeta$ is light-like. In this case, $\tilde M$ lies on a degenerated plane $\widetilde \Pi$ of $\Sigma^4_2$. Up to congruency, we may assume
$$\widetilde \Pi=\{(t,x,y,t,0)|t,x,y\in\mathbb R\}.$$
Since $\zeta$ is light-like, we have $\beta_0=\pm\alpha_0$. By replacing $e_3$ with $-e_3$ if necessary, we may assume $\beta_0=\alpha_0$. Thus, equation \eqref{Shatf3f5} implies $\tilde{S}_{f_{3}}=\tilde{S}_{f_{5}}=\alpha_0I$ which yields that $\tilde M$ is a flat, pseudo-umbilical Lorentzian surface with parallel mean curvature vector. A direct computation yields that $\tilde M$ is congruent to the surface given in the case \ref{LabCaseDegenM21} of the proposition.

\textbf{Case II}. $\delta_4=-1$. In this case,  $\tilde M$ lies on a Lorentzian hyperplane $\Sigma^4_1$ of $\mathbb E^5_2$ and its induced metric is either Riemannian or Lorentzian subject to $\delta_3=-1$ or $\delta_3=1$, respectively. 

\textbf{Case IIa}. $\delta_4=-1$, $\delta_3=1$. In this case, by a similar way to Case Ia, we obtain that $\tilde M$ is a Lorentzian isoparametric surface lying on $\Pi^3_1\simeq\mathbb E^3_1$. Thus, we have $\tilde M=\mathbb S^2_1(r^2)$ which gives gain the case \ref{LabCaseS21}.

\textbf{Case IIb}. $\delta_4=-1$, $\delta_3=-1$. In this case, the induced metric of $\tilde M$ is Riemannian. Moreover, similar to  Case Ib, we have  two subcases regarding to causality of $\zeta$.

\textbf{Case IIb.(i)} $\zeta$ is not light-like. In this case, similar to Case Ib(i), we obtain the case \ref{LabCaseS20} or the case \ref{LabCaseH20}, if $\zeta$ is time-like or space-like, respectively.

\textbf{Case IIb.(ii)} $\zeta$ is light-like. In this case, by a similar way to Case Ib(ii), we see that $M$ is congruent to  the surface given in the case \ref{LabCaseDegenM20} of the proposition.
\end{proof}

%%%%%%%%%%%%%%%%%%%%%%%%%%%%%%%%%%%%%%%%%%%%%%%%%%%%%%%%%%%%%%%%%%%%%%%%%%%%%%%%%%%%%%%%%%%%%%%%%%%%%%%%%%%%%%%%%%%%%%%%%%%%%%%%%%%%%%%%%%%%%%%%%%%%%%%%%%%%%%%%%%%%%%%%%%%%%%%%%%%%%%%%%%%%%%%%%%%%%%%%%%%%%%%%%%%%%%%%%%%%%%%%%%%%%%%%%%%%%%%%%%%%%%%%%%%%%%%%%%%%%%%%%%%%%%%%%%%%%%%%%%%%%%%%%%%%%%%%%%%%%%%%%%%%%%%%%%%%%%%%%%%%%%%%%%%%%%%%%%%%%%%%%%%%%%%%%%%%%%%%%%%%%%%%%%%%%%%%%%%%%%%%%%%%%%%%%%%%%%%%%%%%%%%%%%%%%%%%%%%%%%%%%%%%%%%%%%%%%%%%%%%%%%%%%%%%%%%%%%%%%%%%%%%%%%%%%%%%%%%%%%%%%%%%%%%%%%%%%%%%%%%%%%%%%%%%%%%%%%%%%%%%%%%%%%%%%%%%%%%%%%%%%%%%%%%%%%%%%%%%%%%%%%%%%%%%%%%%%%%%%%%%%%%%%%%%%%%%%%%%%%%%%%%%%%%%%%%%%%%%%%%%%%%%%%%%%%%%%%%%%%%%%%%%%%%%%%%%%%%%%%%%%%%%%%%%%%%%%%%%%%%%%%%%%%%%%%%%%%%%%%%%%%%%%%%%%%%%%%%%%%%%%%%%%%%%%%%%%%%%%%%%%%%%%%

\begin{Lemma}\label{INTCURVEOFE4}
Let $M$ be a biconservative  hypersurface of index 2 in $\mathbb E^{5}_{2}$ with the diagonalizable shape operator and $D': x(t_{0}, s_{0}, u_{0}, v) = \gamma (v)$ be the integral curve of $e_{4}$ passing through a point $p$ of $M$. Then, we have one and only one of the following cases for some constants $R>0,a\neq0$.
\renewcommand{\theenumi}{(\Alph{enumi})}
\begin{enumerate}
\item\label{CASESTRLINE} $k_4=0$ on a neighbourhood of $p$ in $M$ and $\gamma$ is an open part of a line;
\item\label{CASEcircV1} $\varepsilon_1=\pm1$, $\varepsilon_4=1$ and $\gamma$ is congruent to the circle $\frac 1R(0,0,\cos Rv,\sin Rv,0)$;
\item\label{CASEhyperV1} $\varepsilon_1=\pm1$, $\varepsilon_4=-1$ and $\gamma$ is congruent to the hyperbola $\frac 1R(\mathrm{sinh} Rv,0,0,0,\mathrm{cosh} Rv)$;
\item\label{CASEhyperV2} $\varepsilon_1=-1$, $\varepsilon_4=1$ and $\gamma$ is congruent to the hyperbola $\frac 1R(\mathrm{cosh} Rv,0,0,0,\mathrm{sinh} Rv)$;
\item\label{CASEDegeneratednv1} $\varepsilon_1=-1$, $\varepsilon_4=1$ and $\gamma$ is congruent to the curve $(av^2,0,v,0,av^2)$;
\item\label{CASEcircV2} $\varepsilon_1=-1$, $\varepsilon_4=-1$ and $\gamma$ is congruent to the circle $\frac 1R(\cos Rv,\sin Rv,0,0,0)$;
\item\label{CASEDegeneratednv2} $\varepsilon_1=-1$, $\varepsilon_4=-1$ and $\gamma$ is congruent to the curve $(av^2,v,0,0,av^2)$.
\end{enumerate}
\end{Lemma}

\begin{proof}
Let $\gamma (v)$ be an integral curve of $e_{4}$, i.e, $(e_4)_{\gamma}=\gamma'$. Then, $\left.k_{4}\right|_\gamma=\beta_0$ and $\left.\omega_{14}(e_4)\right|_\gamma=\alpha_0$ for some constants $\alpha_0,\beta_0$ because of Lemma \ref{LemmaPrCurvatures}. Moreover, using equations \eqref{EQ317B} and \eqref{EQ317D}, we get
\begin{subequations}\label{GammaEqs1ALL}
\begin{eqnarray}
\label{GammaEq1} \gamma''(v)=\left(\widetilde\nabla_{t}t\right)_{\gamma(v)}&=&-\varepsilon_1 \alpha_0\left(e_1\right)_{\gamma(v)}+ \varepsilon_4\beta_0 N_{\gamma(v)},\\
\label{GammaEq1b} \left(\widetilde\nabla_{t}e_1\right)_{\gamma(v)}&=&\varepsilon_4\alpha_0 t,\\
\label{GammaEq1c} \left(\widetilde\nabla_{t}N\right)_{\gamma(v)}&=&-\beta_0 t,
\end{eqnarray}
\end{subequations}
where we put $t=\gamma'=\left.e_4\right|_{\gamma}$ as the unit tangent vector field of $\gamma$. If $k_4=0$ on a neighbourhood of $p$, then we have $\gamma''=0$ which implies the case \ref{CASESTRLINE} of the lemma. Therefore, we consider the case that $\gamma''\neq0$. Now, we have three cases subject to causality of $\gamma''$. 

\textbf{Case I.} $\gamma''$ is space-like. In this case, equation \eqref{GammaEq1} gives
\begin{equation}\nonumber
\widetilde\nabla_{t}t=(\varepsilon_1 \alpha_0^2+ \beta_0^2)^{1/2} n,
\end{equation}
where $n$ is the unit normal vector field of $\gamma$. By a direct computation using equation \eqref{EQ317B} for $A=4$ and $\widetilde\nabla_{e_4}N=-k_4e_4$, we obtain
\begin{equation}\nonumber
\left(\widetilde\nabla_{t}n\right)=-\varepsilon_4 (\varepsilon_1 \alpha_0^2+ \beta_0^2)^{1/2}t.
\end{equation}
Therefore, $\gamma$ is a planar curve with constant curvature. Hence, it is either a hyperbola or a circle regarding whether $\varepsilon_4=-1$ or $\varepsilon_4=1$, respectively. Therefore we have either the case \ref{CASEcircV1} or the case \ref{CASEhyperV1} of the lemma for some $R>0$.

\textbf{Case II.} $\gamma''$ is time-like. In this case, we have $\varepsilon_1=-1$ and equation \eqref{GammaEq1} gives that
\begin{equation}\nonumber
\widetilde\nabla_{t}t=(\alpha_0^2-\beta_0^2)^{1/2} n.
\end{equation}
A similar arguement to Case I implies case  \ref{CASEcircV2} or the case \ref{CASEhyperV2} of the lemma for some $R>0$ subject to $\varepsilon_4=-1$ or $\varepsilon_4=1$, respectively.

\textbf{Case III.} $\gamma''$ is light-like. In this case  we have $\varepsilon_1=-1$ and we may assume $\alpha=\beta=a$ by replacing $N$ with $-N$ if necessary. Thus, equation \eqref{GammaEq1} gives that
\begin{equation}\label{GammaEq4}
\widetilde\nabla_{t}t=a\left(\left(e_1\right)_{\gamma(v)}+ \varepsilon_4 N_{\gamma(v)}\right).
\end{equation}
A further computation using equations \eqref{GammaEq1b} and \eqref{GammaEq1c}, we obtain $\widetilde\nabla_{t}\gamma''=0$. Thus, $\gamma''$ is a constant, light-like vector. Up to isometries of $\mathbb E^5_2$, we assume $\gamma''=a(1,0,0,0,1)$. By integrating this equation, we obtain the case \ref{CASEDegeneratednv1} and \ref{CASEDegeneratednv2} of the lemma.
\end{proof}

\subsection{Classification Theorems}
In this section, we obtain \textit{local} parametrization of biconservative hypersurfaces with 3 distinct principal curvatures. We would like to mention that, in the theorems obtained, it is assumed that the gradient of the mean curvature vector $H$ of $M$ is not light-like.
\begin{theorem}\label{MAINTHM1}
Let $M$ be an oriented biconservative hypersurface of index 2 in the pseudo-Euclidean space $\mathbb E^5_2$. Assume that its shape operator has the form 
$$S=\mathrm{diag}(k_1,0,0,k_4),\quad k_4\neq0.$$
Then, it is congruent to one of the following eight type of generalized cylinders over surfaces for some smooth functions $\phi=\phi(s)$ and $\psi=\psi(s)$.
\renewcommand{\theenumi}{(\roman{enumi})}
\begin{enumerate}
\item\label{CylinderVersion1Case1} $x(s, t, u, v) = (t, u,\phi\cos v,\phi\sin v,\psi),\quad \phi'^2+\psi'^2=1$;

\item\label{CylinderVersion1Case2} $x(s, t, u, v) = (\phi\mathrm{sinh} v,t, u,\phi\mathrm{cosh} v,\psi),\quad \phi'^2+\psi'^2=1$;

\item\label{CylinderVersion1Case3} $x(s, t, u, v) = (\psi,t, u,\phi\cos v,\phi\sin v),\quad \phi'^2-\psi'^2=-1$;

\item\label{CylinderVersion1Case4} $x(s, t, u, v) = (\phi\mathrm{cosh} v,t, u,\phi\mathrm{sinh} v,\psi),\quad \phi'^2-\psi'^2=1$;

\item\label{CylinderVersion1Case5} $\displaystyle x(s, t, u, v) = \left( \frac {v^2s}2+\psi+s,t,u,vs,\frac {v^2s}2+\psi\right), \ 1-2\psi'<0$;

\item\label{CylinderVersion1Case6} $x(s, t, u, v) = (\phi\cos v,\phi\sin v,t,u,\psi),\quad \phi'^2-\psi'^2=1$;

\item\label{CylinderVersion1Case7} $x(s, t, u, v) = (\phi\mathrm{sinh}v,\psi,t,u,\phi\mathrm{cosh}v),\quad \phi'^2-\psi'^2=-1$;

\item\label{CylinderVersion1Case8} $\displaystyle x(s, t, u, v) = \left(\frac {sv^2}2+\psi,sv,t,u,\frac {sv^2}2+\psi+s\right),\ 1+2\psi'<0$.
\end{enumerate}
\end{theorem}
\begin{proof}
Let $k_2$ vanishes identically on $M$ and $\tilde M$ be the integral submanifold of the distribution $D$ passing through a point $p=x(0,0,0,0)\in M$. Now, consider the local parametrization $x(s,t,u,v)$ given in equation \eqref{LOCALARAMETRIZ} for some smooth functions $\phi_{1},\phi_{2}$ and smooth mappings $\Theta_{1}, \Theta_{2}, \Gamma=(\Gamma_1,\Gamma_2,\Gamma_3,\Gamma_4,\Gamma_5)$.

 Then, from Proposition \ref{PROPINSUBMOFD}, we see that  $\tilde M$ is a 2-plane. Thus, up to isometries of $\mathbb E^5_2$,
we may assume that $y(t,u)=x(0,t,u,0)$ is congruent to $(t,u,0,0,0)$ or $(0,t,u,0,0)$ or $(0,0,0,t,u)$. Furthermore, by redefining $t,u,\phi_2,\Gamma$ appropriately, we may assume that $\Theta_{1}(t, u)=y(t,u)$ and $\phi_1=1$. We also put $\phi_2=\phi$.

\textbf{Case I}. Let us consider $y(t,u)=(t,u,0,0,0)$. In this case, we have $\varepsilon_1=\varepsilon_4=1$. Thus, by the Lemma \ref{INTCURVEOFE4}, the integral curve of $e_4$ is a circle on a Riemannian plane and equation \eqref{LOCALARAMETRIZ} becomes
\begin{equation}\label{LOCALARAMETRIZCylinderVersion1Case1Eq1}
x(s, t, u, v) = (t, u,0,0,0) + \phi(s)\Theta_{2}(v) + \Gamma(s).
\end{equation} 
Now from assumption, $e_1$ and $e_4$ are space-like vectors. Thus, by redefining $\phi$ and $\Gamma$, we may assume that $\Theta_2$ is position vector of a circle of radius 1 with center at origin. Furthermore, since $\langle x_t,x_v\rangle=\langle x_u,x_v\rangle=0$, by redefining $\Gamma$ appropriately if necessary and applying an isometry of $\mathbb E^5_2$, we may assume that
$\Theta_{2}(v)=(0,0,\cos v,\sin v,0).$
Therefore, equation \eqref{LOCALARAMETRIZCylinderVersion1Case1Eq1} becomes
$x(s, t, u, v) = (t, u,\phi(s)\cos v,\phi(s)\sin v,0)+ \Gamma(s).$
Considering $\{x_s,x_t,x_u,x_v\}$ as an orthonormal base field, we obtain the case \ref{CylinderVersion1Case1} of the theorem.

\textbf{Case II}. In this case, $y(t,u)=(0,t,u,0,0)$. Hence equation  \eqref{LOCALARAMETRIZ} becomes
\begin{equation}\label{LOCALARAMETRIZCylinderVersion1Case2Eq1}
x(s, t, u, v) = (0,t, u,0,0) + \phi(s)\Theta_{2}(v) + \Gamma(s)
\end{equation} 
and we have $-\varepsilon_2=\varepsilon_3=1$. Therefore, we have two subcases:

\textbf{Case IIA}. Firstly, we have $\varepsilon_1=1$ and $ \varepsilon_4=-1$. In this case, Lemma \ref{INTCURVEOFE4} implies that the integral curve of $e_4$ is congruent to the hyperbola $\frac 1R(\mathrm{sinh} Rv,0,0,\mathrm{cosh} Rv)$. By the same way to the case I, we have the case   \ref{CylinderVersion1Case2} of the theorem.

\textbf{Case IIB}. Secondly, we have $\varepsilon_1=-1$ and $\varepsilon_4=1$. In this case, the integral curve of $e_4$ is  congruent to  the circle $\frac 1R(0,0,\cos Rv,\sin Rv,0)$ or the hyperbola $\frac 1R(\mathrm{cosh} Rv,0,0,0,\mathrm{sinh} Rv)$ or the curve $(av^2,0,v,0,av^2)$. If it is congruent to the circle or hyperbola, we have either case \ref{CylinderVersion1Case3} or \ref{CylinderVersion1Case4} of the theorem, respectively.

Now assume that the integral curve of $e_4$ is congruent to the curve $(av^2,0,v,0,av^2)$. Since $\langle x_t,x_v\rangle=\langle x_u,x_v\rangle=0$, we have $\Theta_2(v)=(\theta_1(v),c_2,c_3,\theta_4(v),\theta_5(v))$ for some constants $c_2$ and $c_3$. By redefining $\Gamma$ appropriately, we may assume that $c_2=c_3=0$. Therefore, up to isometries of $\mathbb E^5_2$, we may assume that $\Theta_{2}(v)=(av^2,0,0,v,av^2).$ Thus, equation \eqref{LOCALARAMETRIZCylinderVersion1Case2Eq1} becomes
\begin{equation}\label{LOCALARAMETRIZCylinderVersion1Case2Eq2}
x(s, t, u, v) = (av^2\phi+ \Gamma_1,t+ \Gamma_2, u+ \Gamma_3,v\phi+ \Gamma_4,av^2\phi+ \Gamma_5). 
\end{equation} 
Considering $\langle x_s,x_u\rangle=\langle x_s,x_t\rangle=\langle x_s,x_v\rangle=0$, we get $ \Gamma_2'= \Gamma_3'=\Gamma_4'=0$ and $\phi=2a(\Gamma_1-\Gamma_5)+c_1$ for a constant $c_1$. Now, up to a translation, we may assume that $ \Gamma_2= \Gamma_3=\Gamma_4=0$.  Thus,  from equation \eqref{LOCALARAMETRIZCylinderVersion1Case2Eq2}, we get
\begin{align}\label{LOCALARAMETRIZCylinderVersion1Case2Eq3}
\begin{split}
x(s, t, u, v) = \Big(&a v^2 \left(2 a \left(\Gamma _1-\Gamma _5\right)+c_1\right)+\Gamma _1,t,u,v \left(2 a \left(\Gamma _1-\Gamma _5\right)+c_1\right),\\&
a v^2 \left(2 a \left(\Gamma _1-\Gamma _5\right)+c_1\right)+\Gamma_5\Big)
\end{split} 
\end{align} 
Further, defining new coordinates $\tilde s,\tilde v$ by $\tilde s=\Gamma_1-\Gamma_5$, $\tilde v=2a v$, we obtain that $M$ is congruent to the surface given in the case \ref{CylinderVersion1Case5} of the theorem.

\textbf{Case III}. Let us assume that $y(t,u)=(0,0,t,u,0)$. In this case, we have $\varepsilon_2=\varepsilon_3=1$. Therefore, we have $\varepsilon_1=\varepsilon_4=-1$. Furthermore, because of Lemma \ref{INTCURVEOFE4}, the integral curve of $e_4$ is congruent to  the hyperbola $\frac 1R(\mathrm{sinh} Rv,0,0,0,\mathrm{cosh} Rv)$ or the circle $\frac 1R(\cos Rv,\sin Rv,0,0,0)$ or to the curve $(av^2,v,0,0,av^2)$. By a similar way to the Case IIB, we obtain case \ref{CylinderVersion1Case6}-\ref{CylinderVersion1Case8} of the theorem.
\end{proof}
%%%%%%%%%%%%%%%%%%%%%%%%%%%%%%%%%%%%%%%%%%%%%%%%%%%%%%%%%%%%%%%%%%%%%%%%%%%%%%%%%%%%%%%%%%%%%%%%%%%%%%%%%%%%%%%%%%%%%%%%%%%%%%%%%%%%%%%%%%%%%%%%%%%%%%%%%%%%%%%%%%%%%%%%%%%%%%%%%%%%%%%%%%%%%%%%%%%%%%%%%%%%%%%%%%%%%%%%%%%%%%%%%%%%%%%%%%%%%%%%%%%%%%%%%%%%%%%%%%%%%%%%%%%%%%%%%%%%%%%%%%%%%%%%%%%%%%%%%%%%%%%%%%%%%%%%%%%%%%%%%%%%%%%%%%%%%%%%%%%%%%%%%%%%%%%%%%%%%%%%%%%%%%%%%%%%%%%%%%%%%%%%%%%%%%%%%%%%%%%%%%%%%%%%%%%%%%%%%%%%%%%%%%%%%%%%%%%%%%%%%%%%%%%%%%%%%%%%%%%%%%%%%%%%%%%%%%%%%%%%%%%%%%%%%%%%%%%%%%%%%%%%%%%%%%%%%%%%%%%%%%%%%%%%%%%%%%%%%%%%%%%%%%%%%%%%%%%%%%%%%%%%%%%%%%%%%%%%%%%%%%%%%%%%%%%%%%%%%%%%%%%%%%%%%%%%%%%%%%%%%%%%%%%%%%%%%%%%%%%%%%%%%%%%%%%%%%%%%%%%%%%%%%%%%%%%%%%%%%%%%%%%%%%%%%%%%%%%%%%%%%%%%%%%%%%%%%%%%%%%%%%%%%%%%%%%%%%%%%%%%%%%%%%%%%%%%%%%%%%%%%%%%%%%%%%%%%%%%%%%%%%%%%%%%%%%%%%%%%%%%%%%%%%%%%%%%%%%%%%%%%%%%%%%%%%%%%%%%%%%%%%%%%%%%%%%%%%%%%%%%%%%%%%%%%%%%%%%%%%%%%%%%%%%%%%%%%%%%%%%%%%%%%%%%%%%%%%%%%%%%%%%%%%%%%%%%%%%%%%%%%%%%%%%%%%%%%%%%%%%%%%%%%%%%%%%%%%%%%%%%%%%%%%%%!
%%%%%%%%%%%%%%%%%%%%%%%%%%%%%%%%%%%%%%%%%%%%%%%%%%%%%%%%%%%%%%%%%%%%%%%%%%%%%%%%%%%%%%%%%%%%%%%%%%%%%%%%%%%%%%%%%%%%%%%%%%%%%%%%%%%%%%%%%%%%%%%%%%%%%%%%%%%%%%%%%%%%%%%%%%%%%%%%%%%%%%%%%%%%%%%%%%%%%%%%%%%%%%%%%%%%%%%%%%%%%%%%%%%%%%%%%%%%%%%%%%%%%%%%%%%%%%%%%%%%%%%%%%%%%%%%%%%%%%%%%%%%%%%%%%%%%%%%%%%%%

By an exactly same way with the proof of Theorem \ref{MAINTHM1}, we obtain the following theorem.
\begin{theorem}\label{MAINTHM2}
Let $M$ be an oriented hypersurface of index 2 in the pseudo-Euclidean space $\mathbb E^5_2$. Assume that its shape operator has the form 
$$S=\mathrm{diag}(k_1,k_2,k_2,0),\quad k_2\neq0.$$
Then, it is congruent to one of the following eight type of cylinders for some smooth functions $\phi=\phi(s)$ and $\psi=\psi(s)$.
\renewcommand{\theenumi}{(\roman{enumi})}
\begin{enumerate}
\item\label{NonCylinderCase1} $x(s,t,u,v)=(v, \phi\mathrm{cosh} t, \phi\mathrm{sinh} t \cos u, \phi\mathrm{sinh} t \sin u, \psi),\quad \phi'^2-\psi'^2=1$;

\item\label{NonCylinderCase2} $x(s,t,u,v)=(v,\psi, \phi\cos t, \phi\sin t \cos u,\phi\sin t \sin u),\quad \phi'^2-\psi'^2=-1$;

\item\label{NonCylinderCase3} $x(s,t,u,v)=(\phi\mathrm{cosh} t \sin u, \phi\mathrm{cosh} t \cos u, \phi\mathrm{sinh} t,\psi,v),\quad \phi'^2-\psi'^2=1$;

\item\label{NonCylinderCase4} $x(s,t,u,v)=(\psi, \phi\mathrm{sinh} t, \phi\mathrm{cosh}t\cos u, \phi\mathrm{cosh} t \sin u, v),\quad \phi'^2-\psi'^2=-1$;

\item\label{NonCylinderCase5} $x(s,t,u,v)=(v, \phi\mathrm{sinh} t, \phi\mathrm{cosh}t\cos u, \phi\mathrm{cosh} t \sin u, \psi),\quad \phi'^2+\psi'^2=1$;

\item\label{NonCylinderCase6} $x(s,t,u,v)=(\phi\mathrm{sinh} t cos u, \phi\mathrm{sinh} t \sin u, \phi\mathrm{cosh} u, \psi, v),\quad \phi'^2+\psi'^2=1$;

\item\label{NonCylinderCase7}$\displaystyle x(s, t, u, v) = \left(\frac{s (t^2+u^2)}{2}+\psi ,v,st,su,\frac{s (t^2+u^2)}{2}+\psi-s\right),\ \ 1-2\psi'<0$;

\item\label{NonCylinderCase8}$\displaystyle x(s, t, u, v) = \left(\frac{s (t^2-u^2)}{2}+\psi ,st,su,v,\frac{s (t^2-u^2)}{2}+\psi+s\right),\ \ 1+2\psi'<0$.

\end{enumerate}
\end{theorem}

%%%%%%%%%%%%%%%%%%%%%%%%%%%%%%%%%%%%%%%%%%%%%%%%%%%%%%%%%%%%%%%%%%%%%%%%%%%%%%%%%%%%%%%%%%%%%%%%%%%%%%%%%%%%%%%%%%%%%%%%%%%%%%%%%%%%%%%%%%%%%%%%%%%%%%%%%%%%%%%%%%%%%%%%%%%%%%%%%%%%%%%%%%%%%%%%%%%%%%%%%%%%%%%%%%%%%%%%%%%%%%%%%%%%%%%%%%%%%%%%%%%%%%%%%%%%%%%%%%%%%%%%%%%%%%%%%%%%%%%%%%%%%%%%%%%%%%%%%%%%%%%%%%%%%%%%%%%%%%%%%%%%%%%%%%%%%%%%%%%%%%%%%%%%%%%%%%%%%%%%%%%%%%%%%%%%%%%%%%%%%%%%%%%%%%%%%%%%%%%%%%%%%%%%%%%%%%%%%%%%%%%%%%%%%%%%%%%%%%%%%%%%%%%%%%%%%%%%%%%%%%%%%%%%%%%%%%%%%%%%%%%%%%%%%%%%%%%%%%%%%%%%%%%%%%%%%%%%%%%%%%%%%%%%%%%%%%%%%%%%%%%%%%%%%%%%%%%%%%%%%%%%%%%%%%%%%%%%%%%%%%%%%%%%%%%%%%%%%%%%%%%%%%%%%%%%%%%%%%%%%%%%%%%%%%%%%%%%%%%%%%%%%%%%%%%%%%%%%%%%%%%%%%%%%%%%%%%%%%%%%%%%%%%%%%%%%%%%%%%%%%%%%%%%%%%%%%%%%%%%%%%%%%%%%%%%%%%%%%%%%%%%%%%%%%%%%%%%%%%%%%%%%%%%%%%%%%%%%%%%%%%%%%%%%%%%%%%%%%%%%%%%%%%%%%%%%%%%%%%%%%%%%%%%%%%%%%%%%%%%%%%%%%%%%%%%%%%%%%%%%%%%%%%%%%%%%%%%%%%%%%%%%%%%%%%%%%%%%%%%%%%%%%%%%%%%%%%%%%%%%%%%%%%%%%%%%%%%%%%%%%%%%%%%%%%%%%%%%%%%%%%%%%%%%%%%%%%%%%%%%%%%%%%%%!
%%%%%%%%%%%%%%%%%%%%%%%%%%%%%%%%%%%%%%%%%%%%%%%%%%%%%%%%%%%%%%%%%%%%%%%%%%%%%%%%%%%%%%%%%%%%%%%%%%%%%%%%%%%%%%%%%%%%%%%%%%%%%%%%%%%%%%%%%%%%%%%%%%%%%%%%%%%%%%%%%%%%%%%%%%%%%%%%%%%%%%%%%%%%%%%%%%%%%%%%%%%%%%%%%%%%%%%%%%%%%%%%%%%%%%%%%%%%%%%%%%%%%%%%%%%%%%%%%%%%%%%%%%%%%%%%%%%%%%%%%%%%%%%%%%%%%%%%%%%%%%%%%%%%%%%%%%%%%%%%%%%%%%%%%%%%%%%%%%%%%%%%%%%%%%%%%%%%%%%%%

In the next theorem, we obtain local parametrizations of biconservative hypersurfaces with 3 distinct non-zero  principal curvatures.
\begin{theorem}\label{MAINTHM3}
Let $M$ be an oriented hypersurface of index 2 in the pseudo-Euclidean space $\mathbb E^5_2$. Assume that its shape operator has the form 
$$S=\mathrm{diag}(k_1,k_2,k_2,k_4),\quad k_4\neq k_2$$
for some non-vanishing smooth functions $k_1,k_2,k_4$.
Then, it is congruent to one of the following eight type of hypersurfaces for some smooth functions $\phi_{1}=\phi_{1}(s)$ and $\phi_{2}=\phi_{2}(s)$.
\renewcommand{\theenumi}{(\roman{enumi})}
\begin{enumerate}
\item\label{HYPERSURFVersion2Case1} $x(s,t,u,v)=\left(\phi_{2}\mathrm{sinh} v, \phi_{1}\mathrm{cosh} t, \phi_{1}\mathrm{sinh} t \cos u, \phi_{1}\mathrm{sinh} t \sin u, \phi_{2}\mathrm{cosh} v\right),\quad \phi_1'^2-\phi_2'^2=1;$

\item\label{HYPERSURFVersion2Case2}  $x(s,t,u,v)=\left(\phi_{2}\cos v, \phi_{2}\sin v, \phi_{1}\cos t, \phi_{1}\sin t \cos u, \phi_{1}\sin t \sin u\right),\quad \phi_1'^2-\phi_2'^2=-1;$
\item\label{HYPERSURFVersion2Case3} $x(s,t,u,v)=\left(\phi_{1}\mathrm{cosh} t \sin u, \phi_{1}\mathrm{cosh} t \cos u, \phi_{1}\mathrm{sinh} t, \phi_2\cos v, \phi_2\sin v\right),\quad \phi_1'^2-\phi_2'^2=1;$

\item\label{HYPERSURFVersion2Case5} $x(s,t,u,v)=\left(\phi_2\mathrm{sinh} v, \phi_1\mathrm{sinh} t, \phi_1\mathrm{cosh}t\cos u, \phi_1\mathrm{cosh} t \sin u, \phi_2\mathrm{cosh} v\right),\quad \phi_1'^2+\phi_2'^2=1;$

\item\label{HYPERSURFVersion2Case4} $x(s,t,u,v)=\left(\phi_2\mathrm{cosh} v, \phi_1\mathrm{sinh} t, \phi_1\mathrm{cosh}t\cos u, \phi_1\mathrm{cosh} t \sin u, \phi_2\mathrm{sinh} v\right),\quad \phi_1'^2-\phi_2'^2=-1;$

\item\label{HYPERSURFVersion2Case6} $x(s,t,u,v)=\left(\phi_1\mathrm{sinh} t cos u, \phi_1\mathrm{sinh} t \sin u, \phi_1\mathrm{cosh} u, \phi_2\cos v,\phi_2\sin v\right),\quad \phi_1'^2+\phi_2'^2=1;$

\item\label{HYPERSURFVersion2CaseDegS1} A hypersurface given by
\begin{align}\label{HYPERSURFVersion2CaseDegS1Eq1}
\begin{split}
x( s, t, u, v)=&\left(\frac{ s }{2}\left( t^2+ u^2- v^2\right)-a v^2 + \psi, v (2a+ s), s  t, s  u,\right.\\
&\ \ \left.\frac{ s}{2} \left( t^2+ u^2- v^2 \right)-a v^2 +\psi -s
\right)
\end{split}
\end{align}
for a non-zero constants $a$ and a smooth function $\psi=\psi ( s)$ such that $1-2 \psi '<0$;

\item\label{HYPERSURFVersion2CaseDegS2} A hypersurface given by
\begin{align}\label{HYPERSURFVersion2CaseDegS2Eq1}
\begin{split}
x( s, t, u, v)=& \left(\frac{ s\left( t^2- u^2- v^2\right)}{2} +a  v^2+\psi, s  t, s  u, v ( s-2a),\right.\\
&\ \  \left.\frac{ s\left( t^2- u^2- v^2\right)}{2} +a  v^2+\psi+s\right) 
\end{split}
\end{align}
for a non-zero constants $a$ and a smooth function $\psi=\psi (\tilde s)$ such that $1+2 \psi '<0$.
\end{enumerate}
\end{theorem}

\begin{proof}
Let $\tilde M$ be the integral submanifold of the distribution $D$ passing through a point $p=x(0,0,0,0)\in M$, where $x=x(s,t,u,v)$ is the local parametrization of $M$ near $p$ given by the equation \eqref{LOCALARAMETRIZ} for some $\mathbb E^{5}_{2}$-valued functions $\Theta_{1}$, $\Theta_{2}$, $\Gamma$ and some smooth real valued functions $\phi_{1}$, $\phi_{2}$. From Proposition \ref{PROPSTUVCOORDS}, we have
\begin{subequations}\label{MainTHMMETRICSEq1All}
\begin{eqnarray}
\label{MainTHMMETRICSEq1a}\langle x_s,x_s\rangle=\varepsilon_1,\quad \langle x_s,x_t\rangle=\langle x_s,x_u\rangle=\langle x_s,x_v\rangle=0,\\
\label{MainTHMMETRICSEq1b}\langle x_t,x_v\rangle=\langle x_u,x_v\rangle=0.
\end{eqnarray}
\end{subequations}

Because of Corollary \ref{COROLINTSUB} and Corollary \ref{COROLINTCURVE}, $y(t,u)=x(0,t,u,0)$ and $\gamma(v)=x(0,0,0,v)$ are integral submanifolds of $M$. By redefining $\phi_{1}$, $\phi_{2}$, $\Gamma$ properly and using an appropriated isometry of $\mathbb E^5_2$, we may assume that $\Theta_{1}=c_1 y$ and $\Theta_{2}=c_2 \gamma$ for any non-zero constant $c_1,c_2$. On the other hand, $y$ is the position vector of one of the surfaces given in the case \ref{LabCaseH20}- \ref{LabCaseS22} of the Proposition  \ref{PROPINSUBMOFD}. We consider these cases separately. 

If $\tilde M$ is congruent to the surface given in the case \ref{LabCaseH20} of the Proposition  \ref{PROPINSUBMOFD}, we may assume
$$\Theta_{1}(t, u) = (0, \mathrm{cosh} t, \mathrm{sinh} t \cos u, \mathrm{sinh} t \sin u, 0).$$
Therefore, we have $\varepsilon_2=\varepsilon_3=1$ which gives $\varepsilon_1=\varepsilon_4=-1$. Moreover, by considering equation \eqref{MainTHMMETRICSEq1b}, we assume that $\gamma$ lies on the Lorentzian plane $\{(a,0,0,0,b)|a,b\in\mathbb R\}$. Now from the Lemma \ref{INTCURVEOFE4}, the integral curve of $e_4$ is congruent to hyperbola $\frac 1R(\mathrm{sinh} Rv,0,0,0,\mathrm{cosh} Rv)$. By a further computation using equation \eqref{MainTHMMETRICSEq1a}, we obtain that $M$ is congruent to the hypersurface given in the case \ref{HYPERSURFVersion2Case1} of the theorem. 

By a similar way, we see that the case \ref{LabCaseS20}, \ref{LabCaseH21}, \ref{LabCaseS21} and \ref{LabCaseS22} of the Proposition  \ref{PROPINSUBMOFD} gives the case \ref{HYPERSURFVersion2Case2}-\ref{HYPERSURFVersion2Case6} of the theorem (See Table \ref{table:kysymys}). 

Now, assume that  $\tilde M$ is congruent to the surface given  in the case \ref{LabCaseDegenM20} of the Proposition  \ref{PROPINSUBMOFD}. So, we may assume $\Theta_1(u,v)=( A t^{2} + A u^{2},0, t, u, A t^{2} + A u^{2})$. Thus, equation \eqref{LOCALARAMETRIZ} becomes
\begin{equation}\label{LOCALARAMETRIZbecomes}
x(s,t,u,v)=\phi_1( A t^{2} + A u^{2},0, t, u, A t^{2} + A u^{2})+\phi_2\Theta_2(v)+\Gamma(s).
\end{equation}
In this case, we have $\varepsilon_2=\varepsilon_3=1$, therefore, $\varepsilon_1=\varepsilon_4=-1$. Considering $\langle x_t,x_v\rangle=\langle x_u,x_v\rangle=0$, if we put
$$\Theta_2(v)=(\theta_1(v),\theta_2(v),\theta_3(v),\theta_4(v),\theta_5(v)),$$
we obtain $\theta_3'(v)=\theta_4'(v)=0$ and $\theta_1'(v)=\theta_5'(v)$. Therefore, by considering Lemma \ref{INTCURVEOFE4}, we see that redefining $\Gamma$ properly, we may assume that $\Theta_2(v)=(Bv^2,v,0,0,Bv^2)$ for  constant $B$. Thus, equation \eqref{LOCALARAMETRIZbecomes} implies that
\begin{equation}\label{LOCALARAMETRIZbecomes2}
x(s,t,u,v)=( A\phi_1 t^{2} + A\phi_1 u^{2}+B\phi_2v^2,v\phi_2, t\phi_1, u\phi_1, A\phi_1 t^{2} + A \phi_1u^{2}+B\phi_2v^2)+\Gamma(s).
\end{equation}
Considering equation \eqref{MainTHMMETRICSEq1a}, we obtain 
$\phi _1=2 \left(A \Gamma _1-A \Gamma _5\right)+A a_1$, $\phi _2=-2 \left(B \Gamma _1-B \Gamma _5\right)+B a_2$
for some constants $a_1,a_2$ and $\Gamma_i'=0,\ i=2,3,4$. Therefore, from equation \eqref{LOCALARAMETRIZbecomes}, we see that $M$ is congruent to the hypersurface given by  
\begin{align}\label{LOCALARAMETRIZbecomes3}
\begin{split}
x(s,t,u,v)=&\Big( 2 A^2 \left(t^2+u^2\right) \left(a_1+\Gamma _1-\Gamma _5\right)-2 B^2 v^2 \left(a_2+\Gamma _1-\Gamma _5\right)+\Gamma _1,\\
&-2 B v \left(a_2+\Gamma _1-\Gamma _5\right),2 A t \left(a_1+\Gamma _1-\Gamma _5\right),2 A u \left(a_1+\Gamma _1-\Gamma _5\right),\\
&2 A^2 \left(t^2+u^2\right) \left(a_1+\Gamma _1-\Gamma _5\right)-2 B^2 v^2 \left(a_2+\Gamma _1-\Gamma _5\right)+\Gamma _5
\Big).
\end{split}
\end{align}
Finally, by defining new coordinates  $\tilde s=\Gamma _1-\Gamma _5+a_1$, $\tilde t=2A t$, $\tilde u=2A u$ and $\tilde v=-2B v$, we see that $M$ is congruent to the surface given in equation \eqref{HYPERSURFVersion2CaseDegS1Eq1} for a function $\psi=\psi(\tilde s)$. It is noted that the induced metric of the surface given by equation \eqref{HYPERSURFVersion2CaseDegS1Eq1} has the form
$g=(1-2 \psi ')d s^2+ s^2d u^2+ s^2d u^2-( s+a)^2d v^2.$
Since $M$ has index 2, we have $1-2 \psi '<0$. Hence, we have the case \ref{HYPERSURFVersion2CaseDegS2} of the theorem.

By a similar way, we see that if $\tilde M$ is congruent to the surface given  in the case \ref{LabCaseDegenM21} of the Proposition  \ref{PROPINSUBMOFD}, then $M$ is congruent to the hypersurface given by equation \eqref{HYPERSURFVersion2CaseDegS2Eq1} which yields the case \ref{HYPERSURFVersion2CaseDegS2} of the theorem. 

\begin{table}
\begin{tabular}{|c| c| c|}
\hline
Integral submanifold of D& Integral curve of $e_4$ & The hypersurface obtained\\
&(given in the Lemma  \ref{INTCURVEOFE4})& (given in the Theorem \ref{MAINTHM3})\\
\hline
Congruent to $\mathbb H^2_0$&The case \ref{CASEhyperV1}& The case \ref{HYPERSURFVersion2Case1} \\
Congruent to $\mathbb S^2_0$&The case \ref{CASEcircV2}& The case \ref{HYPERSURFVersion2Case2} \\
Congruent to $\mathbb H^2_1$&The case \ref{CASEcircV1}& The case \ref{HYPERSURFVersion2Case3} \\
Congruent to $\mathbb S^2_1$&The case \ref{CASEhyperV2}&  The case \ref{HYPERSURFVersion2Case5}\\
Congruent to $\mathbb S^2_1$&The case \ref{CASEhyperV1}& The case \ref{HYPERSURFVersion2Case4}\\
Congruent to $\mathbb S^2_2$&The case \ref{CASEcircV1}&The case \ref{HYPERSURFVersion2Case6} \\
The surface given by \eqref{LabCaseDegenM20SURF}&The case  \ref{CASEDegeneratednv2}&The case  \ref{HYPERSURFVersion2CaseDegS1}\\
The surface given by \eqref{LabCaseDegenM21SURF}&The case  \ref{CASEDegeneratednv1}&The case  \ref{HYPERSURFVersion2CaseDegS2}\\\hline
\end{tabular}
\caption{Hypersurfaces obtained for $k_2=k_3\neq0,k_4\neq0$}
\label{table:kysymys}
\end{table}
\end{proof}
\section{Conclusions}
It is observed that Theorem \ref{MAINTHM1}, Theorem \ref{MAINTHM2} and Theorem \ref{MAINTHM3} provide necessary condition for being biconservative of a hypersurface of index 2 in $\mathbb E^5_2$. However, choosing appropriate functions $\phi,\psi$ or $\phi_1,\phi_2$ appearing in these theorems, one can see that there exists biconservative hypersurfaces belonging to each of these families obtained in the previous section. We also would like to mention that all the biconservative hypersurfaces obtained so far has at most three distinct principal curvatures. 

  In this context, an explicit example of biconservative hypersurface in $\mathbb E^5_2$ with \textbf{four distinct } principal curvatures has been presented. Moreover, particular choices of constants $a$ and $b$ in this example provides the existence of  biconservative hypersurfaces belonging to the hypersurface family given in the case \ref{HYPERSURFVersion2CaseDegS1} of the Theorem \ref{MAINTHM3}.
\begin{Example}
Consider the hypersurface $M$ given by
\begin{align}\label{DegS1Eq1Generlza}
\begin{split}
x(s, t, u, v)=&\left(-a v^2+b u^2+\frac{1}{2} s \left(t^2+u^2-v^2\right)+\psi ,v (s+2a),s t,u (s+2b),\right.\\&\left.\ \ 
-a v^2+b u^2+\frac{1}{2} s \left(t^2+u^2-v^2\right)+\psi -s
\right)
\end{split}
\end{align}
in the pseudo-Euclidean space $\mathbb E^5_2$, where $a\neq0,b$ are  constants and $2 \psi '-1>0$. By a direct computation, we see that vector fields
\begin{align}\nonumber
\begin{split}
e_1=\frac{\nabla H}{(-\langle\nabla H,\nabla H\rangle)^{1/2}}=\frac{1}{\sqrt{2 \psi '-1}}\partial_s,\\
e_2= \frac{1}{s}\partial_t,\quad e_3= \frac{1}{s+2b}\partial_t,\quad e_4= \frac{1}{s+2a}\partial_v.
\end{split}
\end{align}
form an orthonormal frame field for the tangent bundle of $M$ such that $-\varepsilon_1=\varepsilon_2=\varepsilon_3=-\varepsilon_4=1$ and the unit normal vector field of $M$ is given by
\begin{align}\nonumber
\begin{split}
N=\frac{1}{\sqrt{2 \psi '-1}} \left(\frac{t^2+u^2-v^2}2+1-\psi',v,t,u,\frac{t^2+u^2-v^2}2-\psi'\right).
\end{split}
\end{align}
A further computation yields that $e_1,e_2,e_3,e_4$ are principal directions corresponding to principal curvatures $k_1,k_2,k_3,k_4$ given by 
\begin{align}\label{EXPExamplePrCurvs}
\begin{split}
k_1=\frac{\psi ''}{\left(2 \psi '-1\right)^{3/2}},\quad& k_2=-\frac{1}{s \sqrt{2 \psi '-1}},\quad k_3=-\frac{1}{(s+2b) \sqrt{2 \psi '-1}}\\
&k_4= -\frac{1}{(s+2a) \sqrt{2 \psi '-1}}.
\end{split}
\end{align}
Hence, $M$ is a biconservative hypersurface with index 2 if and only if $k_1=-2H$ and which is equivalent to the second order differential equation
$$\frac{3 \psi ''}{2 \psi '-1}=\frac{1}{s+2a}+\frac{1}{s+2b}+\frac{1}{s}.$$
By solving this differential equation, we obtain
$$\psi=\frac s2+c\int_0^s\left(\xi(\xi+2a)(\xi+2b)\right)^{2/3}d\xi$$
for a non-zero constant c.
\end{Example}

%%%%%%%%%%%%%%%%%%%%%%%%%%%%%%%%%%%%%%%%%%%%%%%%%%%%%%%%%%%%%%%%%%%%%%%%%%%%%%%%%%%%%%%%%%%%%%%%%%%%%%%%%%%%%%%%%%%%%%%%%%%%%%%%%%%%%%%%%%%%%%%%%%%%%%%%%%%%%%%%%%%%%%%%%%%%%%%%%%%%%%%%%%%%%%%%%%%%%%%%%%%%%%%%%%%%%%%%%%%%%%%%%%%%%%%%%%%%%%%%%%%%%%%%%%%%%%%%%%%%%%%%%%%%%%%%%%%%%%%%%%%%%%%%%%%%%%%%%%%%%%%%%%%%%%%%%%%%%%%%%%%%%%%%%%%%%%%%%%%%%%%%%%%%%%%%%%%%%%%%%%%%%%%%%%%%%%%%%%%%%%%%%%%%%%%%%%%%%%%%%%%%%%%%%%%%%%%%%%%%%%%%%%%%%%%%%%%%%%%%%%%%%%%%%%%%%%%%%%%%%%%%%%%%%%%%%%%%%%%%%%%%%%%%%%%%%%%%%%%%%%%%%%%%%%%%%%%%%%%%%%%%%%%%%%%%%%%%%%%%%%%%%%%%%%%%%%%%%%%%%%%%%%%%%%%%%%%%%%%%%%%%%%%%%%%%%%%%%%%%%%%%%%%%%%%%%%%%%%%%%%%%%%%%%%%%%%%%%%%%%%%%%%%%%%%%%%%%%%%%%%%%%%%%%%%%%%%%%%%%%%%%

\begin{Remark}
 An obvious extension of this hypersurface in the pseudo-Euclidean space $\mathbb E^{n+1}_2$ of arbitrary dimension is given by 
\begin{align}\label{DegS1Eq1Generliezza}
\begin{split}
x(s, t_1, t_2, \hdots t_{n-1})=&
\left(-a_1t_1^2+a_2t_2^2+\cdots+a_{n-1}t_{n-1}^2  +\frac{s\left(t_2^2+t_3^2+\cdots+t_{n-1}^2-t_1^2\right)}{2}\right.\\&\ \ +\psi,
t_1 (s+2a_1),t_2 (s+2a_2),\hdots,t_{n-1} (s+2a_2),-a_1t_1^2+a_2t_2^2+\cdots\\&\ \ \left.+a_{n-1}t_{n-1}^2 +
\frac{s\left(t_2^2+t_3^2+\cdots+t_{n-1}^2-t_1^2\right)}{2}+\psi -s
\right)
\end{split}
\end{align}
which provides an example of biconservative hypersurface for a particularly chosen smooth function $\psi$. Moreover, if all constants $a_1,a_2,\hdots a_{n-1}$ are distinct, then $M$ has \textbf{$n$ distinct} principal curvatures.
\end{Remark}

\section*{Acknowledgements}
First author is supported by post doctoral scholarship of ``Harish Chandra Research Institute", Department of Atomic Energy, Government of India and the second named author is supported by  T\"UB\.ITAK (Project Name: 'Y\_EUCL2TIP' Project Number: 114F199).

%%%%%%%%%%%%%%%%%%%%%%%%%%%%%%%%%%%%%%%%%%%%%%%%%%%%%%%%%%%%%%%%%%%%%%%%%%%%%%%%%%%%%%%%%%%%%%%%%%%%%%%%%%%%%%%%%%%%%%%%%%%%%%%%%%%%%%%%%%%%%%%%%%%%%%%%%%%%%%%%%%%%%%%%%%%%%%%%%%%%%%%%%%%%%%%%%%%%%%%%%%%%%%%%%%%%%%%%%%%%%%%%%%%%%%%%%%%%%%%%%%%%%%%%%%%%%%%%%%%%%%%%%%%%%%%%%%%%%%%%%%%%%%%%%%%%%%%%%%%%%%%%%%%%%%%%%%%%%%%%%%%%%%%%%%%%%%%%%%%%%%%%%%%%%%%%%%%%%%%%%%%%%%%%%%%%%%%%%%%%%%%%%%%%%%%%%%%%%%%%%%%%%%%%%%%%%%%%%%%%%%%%%%%%%%%%%%%%%%%%%%%%%%%%%%%%%%%%%%%%%%%%%%%%%%%%%%%%%%%%%%%%%%%%%%%%%%%%%%%%%%%%%%%%%%%%%%%%%%%%%%%%%%%%%%%%%%%%%%%%%%%%%%%%%%%%%%%%%%%%%%%%%%%%%%%%%%%%%%%%%%%%%%%%%%%%%%%%%%%%%%%%%%%%%%%%%%%%%%%%%%%%%%%%%%%%%%%%%%%%%%%%%%%%%%%%%%%%%%%%%%%%%%%%%%%%%%%%%%%%%%%%%%%%%%%%%%%%%%%%%%%%%%%%%%%%%%%%%%%%%%%%%%%%%%%%%%%%%%%%%%%%%%%%%%%%%%%%%%%%%%%%%%%%%%%%%%%%%%%%%%%%%%%%%%%%%%%%%%%%%%%%%%%%%%%%%%%%%%%%%%%%%%%%%%%%%%%%%%%%%%%%%%%%%%%%%%%%%%%%%%%%%

\end{document}